\documentclass[a4paper,twoside,11pt,english,intlimits]{article}

\usepackage[utf8]{inputenc}
\usepackage[T1]{fontenc}
\usepackage{amsmath,amsthm,amsfonts,amssymb}
\usepackage[english]{babel}
\usepackage{index}
\usepackage{times,euler,euscript,upgreek,mathabx}
\usepackage{enumerate}
\usepackage{etex,pictex}
\usepackage{fancyhdr,titlesec,url}
\usepackage[all,knot,poly]{xy}
\usepackage{array}
\usepackage{hyperref}
\usepackage{graphicx}
\usepackage[numbers,sort&compress]{natbib}
\usepackage{multirow}
\usepackage{ifthen}
\usepackage{time}
\usepackage{wrapfig}

\titleformat{\section}[block]{\scshape\filcenter\Large}{\thesection.}{.5em}{}
\titleformat{\subsection}[block]{\bfseries\filcenter\large}{\thesubsection.}{.5em}{\medskip}
\titleformat{\subsubsection}[runin]{\bfseries}{\thesubsubsection.}{.5em}{}[.]
\titlespacing{\subsubsection}{0pt}{10pt}{.5em}

\newtheoremstyle{ntheorem}%
	{\topsep}{\topsep}{\itshape}{0pt}{\bfseries}{.}{.5em}%
	{\thmnumber{#2.\hspace{.5em}}\thmname{#1}\thmnote{ (#3)}}
	
\newtheoremstyle{ndefinition}%
	{\topsep}{\topsep}{\normalfont}{0pt}{\bfseries}{.}{.5em}%
	{\thmnumber{#2.\hspace{.5em}}\thmname{#1}\thmnote{ (#3)}}
	
\newtheoremstyle{nremark}%
	{\topsep}{\topsep}{\normalfont}{0pt}{\itshape}{.}{.5em}%
	{\thmnumber{}\thmname{#1}\thmnote{ (#3)}}

\theoremstyle{ntheorem}
  	\newtheorem{theorem}[subsubsection]{Theorem}
  	\newtheorem{proposition}[subsubsection]{Proposition}
	\newtheorem{lemma}[subsubsection]{Lemma}
  	\newtheorem{corollary}[subsubsection]{Corollary}

\theoremstyle{ndefinition}

	\newtheorem{example}[subsubsection]{Example}

\fancypagestyle{plain}{
\fancyhf{}\fancyfoot[LC,RC]{}
\fancyfoot[LE,RO]{$\thepage$}

}

\UseTips
\SelectTips{eu}{11}

\newdir{ >}{{}*!/-10pt/@{>}}
\newdir{ -}{{}*!/-10pt/@{}}
\newdir{> }{{}*!/+10pt/@{>}}

\makeatletter

\xyletcsnamecsname@{dir4{}}{dir{}}
\xydefcsname@{dir4{-}}{\line@ \quadruple@\xydashh@}
\xydefcsname@{dir4{.}}{\point@ \quadruple@\xydashh@}
\xydefcsname@{dir4{~}}{\squiggle@ \quadruple@\xybsqlh@}
\xydefcsname@{dir4{>}}{\Tttip@}
\xydefcsname@{dir4{<}}{\reverseDirection@\Tttip@}

\xydef@\quadruple@#1{%
	\edef\Drop@@{%
		\dimen@=#1\relax
		\dimen@=.5\dimen@
		\A@=-\sinDirection\dimen@
		\B@=\cosDirection\dimen@
		\setboxz@h{%
			\setbox2=\hbox{\kern3\A@\raise3\B@\copy\z@}%
			\dp2=\z@ \ht2=\z@ \wd2=\z@ \box2
			\setbox2=\hbox{\kern\A@\raise\B@\copy\z@}%
			\dp2=\z@ \ht2=\z@ \wd2=\z@ \box2
			\setbox2=\hbox{\kern-\A@\raise-\B@\copy\z@}%
			\dp2=\z@ \ht2=\z@ \wd2=\z@ \box2
			\setbox2=\hbox{\kern-3\A@\raise-3\B@ \noexpand\boxz@}%
			\dp2=\z@ \ht2=\z@ \wd2=\z@ \box2
		}%
		\ht\z@=\z@ \dp\z@=\z@ \wd\z@=\z@ \noexpand\styledboxz@
	}%
}

\xydef@\Tttip@{\kern2pt \vrule height2pt depth2pt width\z@
	\Tttip@@ \kern2pt \egroup
	\U@c=0pt \D@c=0pt \L@c=0pt \R@c=0pt \Edge@c={\circleEdge}%
	\def\Leftness@{.5}\def\Upness@{.5}%
	\def\Drop@@{\styledboxz@}\def\Connect@@{\straight@{\dottedSpread@\jot}}}
	
\xydef@\Tttip@@{%
	\dimen@=.25\dimen@
 	\B@=\cosDirection\dimen@
	\setboxz@h\bgroup\reverseDirection@\line@ \wdz@=\z@ \ht\z@=\z@ \dp\z@=\z@
	{\vDirection@(1,-1)\xydashl@ \xyatipfont\char\DirectionChar}%
	{\vDirection@(1,+1)\xydashl@ \xybtipfont\char\DirectionChar}%
}

\xydef@\ar@form{
	\ifx \space@\next \expandafter\DN@\space{\xyFN@\ar@form}%
	\else\ifx ^\next \DN@ ^{\xyFN@\ar@style}\edef\arvariant@@{\string^}%
	\else\ifx _\next \DN@ _{\xyFN@\ar@style}\edef\arvariant@@{\string_}%
	\else\ifx 0\next \DN@ 0{\xyFN@\ar@style}\def\arvariant@@{0}%
	\else\ifx 1\next \DN@ 1{\xyFN@\ar@style}\def\arvariant@@{1}%
	\else\ifx 2\next \DN@ 2{\xyFN@\ar@style}\def\arvariant@@{2}%
	\else\ifx 3\next \DN@ 3{\xyFN@\ar@style}\def\arvariant@@{3}%
	\else\ifx 4\next \DN@ 4{\xyFN@\ar@style}\def\arvariant@@{4}%
	\else\ifx \bgroup\next \let\next@=\ar@style
	\else\ifx [\next \DN@[##1]{\ar@modifiers{[##1]}}%]
	\else\ifx *\next \DN@ *{\ar@modifiers}%
	\else\addLT@\ifx\next \let\next@=\ar@slide
	\else\ifx /\next \let\next@=\ar@curveslash
	\else\ifx (\next \let\next@=\ar@curveinout %)
	\else\addRQ@\ifx\next \addRQ@\DN@{\ar@curve@}%
	\else\addLQ@\ifx\next \addLQ@\DN@{\xyFN@\ar@curve}%
	\else\addDASH@\ifx\next \addDASH@\DN@{\defarstem@-\xyFN@\ar@}%
	\else\addEQ@\ifx\next \addEQ@\DN@{\def\arvariant@@{2}\defarstem@-\xyFN@\ar@}%
	\else\addDOT@\ifx\next \addDOT@\DN@{\defarstem@.\xyFN@\ar@}%
	\else\ifx :\next \DN@:{\def\arvariant@@{2}\defarstem@.\xyFN@\ar@}%
	\else\ifx ~\next \DN@~{\defarstem@~\xyFN@\ar@}%
	\else\ifx !\next \DN@!{\dasharstem@\xyFN@\ar@}%
	\else\ifx ?\next \DN@?{\ar@upsidedown\xyFN@\ar@}%
	\else \let\next@=\ar@error
	\fi\fi\fi\fi\fi\fi\fi\fi\fi\fi\fi\fi\fi\fi\fi\fi\fi\fi\fi\fi\fi\fi\fi \next@}

\makeatother

\newcommand{\fll}{\longrightarrow}

\newcommand{\fl}{\to}
\newcommand{\dfl}{\Rightarrow}
\newcommand{\dfll}{\xymatrix @C=1.5em {\strut \ar@2 [r] & \strut}}
\newcommand{\tfl}{\Rrightarrow}
\newcommand{\qfl}{\xymatrix@1@C=10pt{\ar@4 [r] &}}

\newcommand{\ofl}[1]{\xymatrix @C=1.5em {\strut \ar [r] ^-{#1} & \strut}}

\newcommand{\opfl}[1]{\xymatrix @C=1.5em {\strut \ar@{->>} [r] ^-{#1} & \strut}}
\newcommand{\odfl}[1]{\xymatrix @C=1.5em {\strut \ar@2 [r] ^-{#1} & \strut}}
\newcommand{\odfll}[1]{\xymatrix @C=2em {\strut \ar@2 [r] ^-{#1} & \strut}}
\newcommand{\otfl}[1]{\xymatrix @C=1.5em {\strut \ar@3 [r] ^-{#1} & \strut}}
\newcommand{\oqfl}[1]{\xymatrix @C=1.5em {\strut \ar@4 [r] ^*+{#1} & \strut}}

\newcommand{\dfln}[1]{\fll}

\newcommand{\rep}[1]{\widehat{#1}}

\renewcommand{\phi}{\varphi}
\renewcommand{\epsilon}{\varepsilon}

\newcommand{\Cr}{\EuScript{C}}

\newcommand{\Lr}{\EuScript{L}}

\newcommand{\Or}{\EuScript{O}}

\newcommand{\Sr}{\EuScript{S}}

\newcommand{\B}{\mathbf{B}}
\newcommand{\C}{\mathbf{C}}

\newcommand{\M}{\mathbf{M}}

\newcommand{\tck}[1]{#1^{\top}}

\renewcommand{\leq}{\leqslant}
\renewcommand{\geq}{\geqslant}

\def\etapes#1{{#1}_{stp}}
\def\support{\mathrm{Supp}}
\def\precmult{\prec_{mul}}
\def\precmultE{\preccurlyeq_{mul}}

\def\res{\cdot}
\def\sigmab{\Sigma(\B_3^+)}
\def\labNF{\psi^\mathrm{NF}}
\def\labQNF{\psi^\mathrm{QNF}}

\newcount\hh
\newcount\mm
\mm=\time
\hh=\time
\divide\hh by 60
\divide\mm by 60
\multiply\mm by 60
\mm=-\mm
\advance\mm by \time
\def\hhmm{\number\hh:\ifnum\mm<10{}0\fi\number\mm}
\begin{document}
\thispagestyle{empty}

\begin{center}

% Titre
\begin{Large}\begin{uppercase}
{Coherence of string rewriting systems}
\end{uppercase}\end{Large}

\vskip+6pt

\begin{Large}\begin{uppercase}
{by decreasingness}
\end{uppercase}\end{Large}

\vskip+8pt

\bigskip\hrule height 1.5pt \bigskip

% Auteurs
\begin{large}\begin{uppercase}
{Cl\'ement Alleaume -- Philippe Malbos}
\end{uppercase}\end{large}

%%%
\vskip+30pt

\begin{small}\begin{minipage}{14cm}
\noindent\textbf{Abstract --}
Squier introduced a homotopical method in order to describe all the relations amongst rewriting reductions of a confluent and terminating string rewriting system. From a string rewriting system he constructed a $2$-dimensional combinatorial complex whose $2$-cells are generated by relations induced by the rewriting rules. When the rewriting system is confluent and terminating, the homotopy of this complex can be characterized in term of confluence diagrams induced by the critical branchings of the rewriting system.
Such a construction is now used to solve coherence problems for monoids using confluent and terminating string rewriting systems.

In this article, we show how to weaken the termination hypothesis in the description of all the relations amongst rewriting reductions. Our construction uses the decreasingness method introduced by van Oostrom. We introduce  the notion of decreasing two-dimensional polygraph and we give sufficient conditions for a decreasing polygraph to be extended in a coherent way. In particular, we show how a confluent and quasi-terminating polygraph can be extended into a coherent presentation.

\end{minipage}\end{small}
\end{center}

\vspace{0.8cm}

\section{Introduction}

At the end of the eighties, using a homological argument, Squier showed that there are finitely presented monoids with a decidable word problem that cannot be presented by a finite convergent (\emph{i.e.}, confluent and terminating) string rewriting system, \cite{Squier87a, Squier87}.
He linked the existence of a finite convergent presentation for a finitely presented monoid to a homological property by showing that the critical branchings of a convergent string rewriting system generate the module of the $2$-homological syzygies of the presentation.
A purely combinatorial approach is then presented in~\cite{Squier94} to the question of whether or not a finitely presented monoid admits a finite convergent presentation. The existence of such a presentation is linked to a finiteness condition of finitely presented monoids, called \emph{finite derivation type}, that extends the properties of being finitely generated and finitely presented.

Beyond the questions of decidability of the word problem and of the existence of finite convergent presentations, the graph-theoretical tools associated to convergent presentations of monoids developped in~\cite{Squier94} were applied to question of coherence problems for monoids (\emph{e.g.}, Artin monoids~\cite{GaussentGuiraudMalbos15} or plactic monoids~\cite{HageMalbos16}) and monoidal categories~\cite{GuiraudMalbos12mscs}. In particular, one of the problems is to compute a \emph{coherent presentation} of a monoid presented by a string rewriting system. Such a presentation extends the generators and the rules by homotopy generators taking into account all the relations amongst the rewriting sequences.
A method is given in \cite{Squier94} to solve this problem from a convergent string rewriting system.  However, in some situations it is difficult to get both confluence and termination on a finite set of generators and a finite set of rules.

In this article, using decreasingness methods from \cite{vanOostrom94}, we show how to weaken the termination hypothesis in the construction of coherent presentations. As an application we show how to extend a confluent and quasi-terminating string rewriting system into a coherent presentation.

\subsubsection*{Squier's two-dimensional complex}
To a string rewriting system $\Sigma_2$ on an alphabet $\Sigma_1$ Squier, Otto and Kobayashi associated in~\cite{Squier94} a $2$-dimensional cellular complex $S(\Sigma)$, defined independently by Kilibarda \cite{Kilibarda97} and Pride \cite{Pride95}.
The complex $S(\Sigma)$ has only one $0$-cell, its $1$-cells are the strings in the free monoid~$\Sigma_1^\ast$ generated by the alphabet $\Sigma_1$ and its $2$-cells are induced by the rewriting rules $\alpha : u\dfl v$ in $\Sigma_2$ and the set $\Sigma_2^-$ of their inverses $\alpha^- : v\dfl u$. 
That is, there is a $2$-cell in $S(\Sigma)$ between each pair of strings with shape $wuw'$ and $wvw'$ such that $\Sigma_2 \sqcup \Sigma_2^-$ contains the relation $u\dfl v$.
This $2$-dimensional complex is extended with \mbox{$3$-cells}, called \emph{Peiffer confluences}, filling all the $2$-spheres of the following form
\[
\xymatrix @R=0.15em @C=2.5em {
& wv_1w'u_2w''
	\ar@2@/^/ [dr] ^-{wv_1w'\alpha_2w''}
\\
wu_1w'u_2w''
	\ar@2@/^/ [ur] ^-{w\alpha_1w'u_2w''}
	\ar@2@/_/ [dr] _-{wu_1w'\alpha_2w''}
&& wv_1w'v_2w''
\\
& wu_1w'v_2w''
	\ar@2@/_/ [ur] _-{w\alpha_1w'v_2w''}
\ar@3 "1,2";"3,2" ^-{}
}
\]
where $\alpha_1 : u_1\dfl v_1$ and $\alpha_2 : u_2 \dfl v_2$ are in $\Sigma_2\sqcup \Sigma_2^-$ and $w$, $w'$ and $w''$ are strings in $\Sigma_1^\ast$.
The Peiffer confluences make homotopic the $2$-cells corresponding to the application of rewriting steps on non-overlapping strings.

A \emph{homotopy basis} of the complex $S(\Sigma)$ is defined as a set $\Sigma_3$ of additional $3$-cells that makes $S(\Sigma)$ aspherical, that is any $2$-dimensional sphere can be ``filled up'' by the $3$-cells of $\Sigma_3$. The presentation~$\Sigma$ is called of finite derivation type (FDT) if it is finite and it admits a finite homotopy basis. 
The FDT property is an invariant property for finitely presented monoids, that is, if $\Sigma$ and~$\Upsilon$ are two finite string rewriting systems that present the same monoid, then $\Sigma$ has FDT if and only if $\Upsilon$ has FDT, \cite{Squier94}.

\subsubsection*{Squier's completion}
Given a convergent string rewriting system $\Sigma$, the set made of one $3$-cell filling a confluence diagram induced by each critical branching forms a homotopy basis of $S(\Sigma)$, \cite{Squier94}. Such a set of $3$-cells is called a \emph{family of generating confluences of $\Sigma$}.
In others words, any diagram defined by two parallel rewriting sequences can be filled up by confluence diagrams induced by the critical branchings and by the Peiffer confluences. This result corresponds to a homotopical version of Newman's Lemma,~\cite{Newman42}.
In particular, when the presentation is finite, it has finitely many critical branchings, hence a finite family of generating confluences. This is a way to prove that finite convergent presentations have FDT,~\cite{Squier94}.

\subsubsection*{Squier's completion without termination}
The above result starts from a convergent presentation and the construction of homotopy bases is made by Noetherian induction. In some situations, it is difficult to get both confluence and termination without adding new generators, as in the case of plactic monoids~\cite{HageMalbos16} or Artin monoids \cite{GaussentGuiraudMalbos15}. Moreover, the addition of new generators implies as much new relations and thus new potentially non confluent critical branchings. For instance, the Artin monoid on the symmetric group $S_2$ is the monoid of braids on three strands $\B_3^+$ generated by two elements $s$ and $t$ and one relation  $sts = tst$. Kapur and Narendran proved that this monoid does not admit a finite convergent presentation with only two generators, \cite{KapurNarendran85}. Note that a finite convergent presentation can be obtained by Knuth-Bendix completion on the presentation with three generators $s$, $t$, $a$ and the two rules $sts\dfl tst$ and $st\dfl a$, where $a$ is a redundant generator.

\subsubsection*{Coherence for quasi-terminating polygraphs}
In this article, we weaken the termination hypothesis and we give a construction of homotopy bases for decreasing and quasi-terminating string rewriting systems. 
The notion of quasi-termination weakens termination in the sense that if there is an infinite rewriting sequence it must contain infinitely many occurrences of the same $1$-cell. In that case, Noetherian induction cannot be used to construct a coherent presentation. For this reason we proceed by using a well-founded labelling on the rewriting system, called the labelling to the quasi-normal form.
For example, the monoid~$\B_3^+$ admits the following confluent and quasi-terminating presentation 
\[
\big\langle s,t \; \big| \; sts \dfl tst, \; tst \dfl sts \big\rangle.
\]
We obtain a homotopy basis of the monoid $\B_3^+$ containing five $3$-cells. This presentation can be homotopically reduced to obtain an empty homotopy basis.

\subsubsection*{Summary of results}
In this work, we use the categorical description of string rewriting systems by $2$-polygraphs, that are recalled in Section 2. 
We introduce the notion of decreasing $2$-polygraph from the corresponding one introduced by van Oostrom for abstract rewriting systems in \cite{vanOostrom94}.
We will use van Oostrom’s decreasingness techniques to prove our main result. However, decreasingness for string rewriting systems needs to take into account the structure of rewriting on strings. In particular, we introduce the notion of Peiffer decreasingness in order to take into account the confluence diagrams induced by application of rewriting steps on non-overlapping strings and the notion of compatibility with contexts for taking into account the contexts of the rules.

In Section 3, we extend Squier's completion known on convergent $2$-polygraphs to decreasing $2$-polygraphs. We define a \emph{Squier's decreasing completion} of  a decreasing $2$-polygraph $\Sigma$ as an extension of $\Sigma$ by the globular extension of loops, containing one 3-cell for each equivalence class of elementary $2$-loop and the globular extension of generating decreasing confluences, containing  a decreasing confluence diagram for each critical branching of $\Sigma$.

Our main result states that a strictly decreasing $2$-polygraph whose labelling is compatible with contexts and Peiffer decreasing can be extended into a coherent presentation, Theorem~\ref{Theorem:MainResult}. As a consequence of this result, we show how to compute a coherent presentation from a confluent and quasi-terminating $2$-polygraph. Finally, we show how our construction generalizes the one given in \cite{Squier94} for convergent rewriting systems and we deduce some homological and homotopical consequences.

\section{Decreasing polygraphs}

In this section, we recall categorical notions used in this work to describe string rewriting systems and relations between rewriting sequences. We refer the reader to \cite{GuiraudMalbos16} for a deeper presentation of these notions.
Then we introduce decreasing $2$-polygraphs from the corresponding notion for abstract rewriting systems introduced by van Oostrom in \cite{vanOostrom94}. 

\subsection{Two-dimensional polygraphs and extended presentations}

\subsubsection{Two-dimensional polygraphs}
A \emph{$1$-polygraph} $\Sigma$ is a directed graph made of a set of $0$-cells $\Sigma_0$, a set of $1$-cells $\Sigma_1$ and source and target maps $s_0, t_0 : \Sigma_1 \fl \Sigma_0$. We denote by $\Sigma_1^\ast$ the free category generated by $\Sigma_1$.
A \emph{globular extension} of the free category $\Sigma_1^\ast$ is a set $\Sigma_2$ equipped with two maps $s_1,t_1: \Sigma_2 \fl \Sigma_1^\ast$ such that, for every $\alpha$ in $\Sigma_2$, the pair
 $(s_1(\alpha),t_1(\alpha))$ is a \emph{$1$-sphere} in the category $\Sigma_1^\ast$, that is, 
$s_0s_1(\alpha)=s_0t_1(\alpha)$ and $t_0s_1(\alpha)=t_0t_1(\alpha)$.
A \emph{$2$-polygraph} is a triple $\Sigma=(\Sigma_0,\Sigma_1,\Sigma_2)$, where $(\Sigma_0,\Sigma_1)$ is a $1$-polygraph and $\Sigma_2$ is a globular extension of $\Sigma_1^\ast$, whose elements are called the \emph{$2$-cells} of the $2$-polygraph.
A presentation of a category $\C$ is a $2$-polygraph such that the quotient of the free category $\Sigma_1^\ast$ by the congruence generated by $\Sigma_2$ is isomorphic to $\C$. Note that a monoid being a category with a single object, is presented in the same way by a $2$-polygraph with only one $0$-cell.

\subsubsection{Free $2$-categories}
Recall that a \emph{$2$-category} (resp. \emph{$(2,1)$-category}) $\Cr$ is a category enriched in category (resp. groupoid). Equivalently, a $(2,1)$-category is a $2$-category in which all $2$-cells are invertible for the $1$-composition. We denote by $\Cr_2$ the set of $2$-cells of $\Cr$ and the \emph{$0$-composition} (resp. \emph{$1$-composition}) of two $2$-cells $f$ and $g$ in $\Cr$ is denoted by $f\star_0 g$, or by $fg$ (resp. $f\star_1 g$) . We will denote by $s_i$ (resp. $t_i$) the \emph{$i$-source map} (resp. \emph{$i$-target map}) defined on $1$-cells and $2$-cells of a $2$-category. A \emph{$2$-sphere} in $\Cr$ is a pair $(f,g)$ of $2$-cells of $\Cr$ such that $s_1(f)=s_1(g)$ and $t_1(f)=t_1(g)$.

Given a $2$-polygraph $\Sigma$, we will denote by $\Sigma_2^\ast$ the free $2$-category generated by~$\Sigma$ and by $\tck{\Sigma}_2$ the free $(2,1)$-category generated by~$\Sigma$, that is the free $2$-category generated by~$\Sigma$ in which all the $2$-cells are invertible.

\subsubsection{Rewriting sequences}
A \emph{rewriting step} with respect to a $2$-polygraph $\Sigma$ is a $2$-cell of $\Sigma_2^\ast$ of the form $u\varphi v$ where $u$ and $v$ are $1$-cells in $\Sigma_1^\ast$ and $\varphi$ is a $2$-cell of $\Sigma_2$. We denote $\etapes{\Sigma}$ the set of rewriting steps of $\Sigma$.
A \emph{rewriting sequence} with respect to $\Sigma$ is a finite or infinite sequence $f_0\cdot f_1 \res \ldots \res f_i \res\, \cdots$, 
where the $f_i$ are rewriting steps such that $t_1(f_i)=s_1(f_{i+1})$ for all $i\geq 0$.
A $1$-cell  $u$ \emph{rewrites into} a $1$-cell~$v$ if there is a rewriting sequence $f_0\res \ldots \res f_n$ such that $s_1(f_0)=u$ and $t_1(f_n)=v$.

For any rewriting sequence $f_0\res f_1 \res \ldots \res f_n$ from $u$ to $v$ there is a corresponding $2$-cell $f_0\star_1 f_1 \star_1 \ldots \star_1 f_n$ in the $2$-category $\Sigma_2^\ast$ with source $s_1(f_0)=u$ and target $t_1(f_n)=v$.
Conversely, any $2$-cell $f$ in the \linebreak $2$-category $\Sigma_2^\ast$ can be decomposed as a composite $f_0 \star_1 \ldots \star_1 f_n$ of rewriting steps. Note that, this decomposition is unique up to Peiffer relations. 

The \emph{length} of a finite rewriting sequence $f$ is the number, denoted by $\ell(f)$, of rewriting steps occurring in the sequence.
Given two $1$-cells $u$ and $v$ such that $u$ can be reduced to $v$, 
the \emph{distance} from $u$ to $v$, denoted by $d(u,v)$, is the length of the shortest rewriting sequence from $u$ to $v$. 

\subsubsection{Support of a $2$-cell}
Let $\Sigma$ be a $2$-polygraph.
Any $2$-cell $f$ in $\Sigma_2^\ast$ can be written as a $1$-composite of finitely many rewriting steps $u_1\varphi_1v_1,\ldots,u_k\varphi_k v_k$, where the $u_i$ and $v_i$ are $1$-cells in $\Sigma_1^\ast$ and $\varphi_i$ is a $2$-cell in $\Sigma_2$. We define the \emph{support} of the $2$-cell $f$ as the multiset, denoted by $\support(f)$, consisting of the $2$-cells $\varphi_i$ occurring in this decomposition. The support is well-defined because any decomposition of $f$ in $\Sigma_2^\ast$ into a $1$-composite of rewriting steps involves the same rewriting steps.  
Note also that any such a decomposition is finite and thus the support of a $2$-cell is a finite multiset. As a consequence, the multiset inclusion is a well-founded order on supports, allowing us to prove some properties by induction on the support of $2$-cells.

\subsubsection{Branchings}
A \emph{(finite) branching} of a $2$-polygraph $\Sigma$ is a pair $(f,g)$ of (finite) rewriting sequences  of $\Sigma$ with a common source $u=s_1(f)=s_1(g)$. 
Such a branching will be denoted by \linebreak $(f,g) : u \dfl (t_1(f),t_1(g))$.
A \emph{confluence} of a $2$-polygraph $\Sigma$ is a pair $(f',g')$ of rewriting sequences  of $\Sigma$ with a common target $v=t_1(f')=t_1(g')$. 
Such a confluence will be denoted by $(f',g') : (t_1(f),t_1(g)) \dfl v$.

A branching $(f,g)$ is \emph{local} (resp. \emph{aspherical}) if $f$ and $g$ are in $\etapes{\Sigma}$ (resp. $f=g$). 
A \emph{Peiffer branching} of $\Sigma$ is a local branching $(fv,ug)$ with source $uv$ where $u,v$ are composable $1$-cells and $f,g$ are in~$\etapes{\Sigma}$. An \emph{overlapping branching} of $\Sigma$ is a local branching that is not aspherical or Peiffer. An overlapping branching is called a \emph{critical branching} if it is minimal for the order $\sqsubseteq$ on local branchings generated by $(f,g) \sqsubseteq (wfw',wgw')$, for any local branching $(f,g)$ composable with $1$-cells $w$ and $w'$ in~$\Sigma_1^\ast$.

\subsubsection{Termination and quasi-termination}
A $2$-polygraph $\Sigma$ is \emph{terminating} if  it  has  no  infinite  rewriting  sequence, that is there is no sequence $(u_n)_{n \in \mathbb{N}}$ of $1$-cells such that for each $n$ in $\mathbb{N}$, there is a rewriting step from $u_n$ to $u_{n+1}$.
In  that  case, every $1$-cell $u$ of $\Sigma_1^\ast$ has at least one normal form $\rep{u}$, that is, there is no rewriting step with source $\rep{u}$. 

Following \cite{Dershowitz87}, we say that a $2$-polygraph $\Sigma$ is \emph{quasi-terminating} if for each sequence $(u_n)_{n \in \mathbb{N}}$ of \linebreak $1$-cells such that for each $n$ in $\mathbb{N}$ there is a rewriting step from $u_n$ to $u_{n+1}$, the sequence $(u_n)_{n \in \mathbb{N}}$ contains an infinite number of occurrences of the same $1$-cell.

Let $\Sigma$ be a $2$-polygraph.
A $1$-cell $u$ of $\Sigma_1^\ast$ is called a \emph{quasi-normal form} if for any rewriting step with source $u$ leading to a 1-cell $v$, there exists a rewriting sequence from $v$ to $u$.
A quasi-normal form of a $1$-cell $u$ is a quasi-normal form $\widetilde{u}$ such that there exists a rewriting sequence from $u$ to $\widetilde{u}$.
If $\Sigma$ is quasi-terminating, any $1$-cell $u$ of $\Sigma_1^\ast$ admits a quasi-normal form. Note that, this quasi-normal form is neither irreducible nor unique in general.

\subsubsection{Example}
Let us consider the $2$-polygraph defined by the following $2$-graph

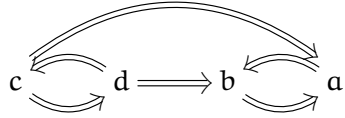
\vspace{0.5cm}
\[
\xymatrix @R=0.15em @C=2.5em {
c
	\ar@2@/_2ex/ [r]
	\ar@2@<+1ex>@/^5ex/ [rrr]
& d
	\ar@2 [r]
	\ar@2@/_2ex/ [l]
& b
    \ar@2@/_2ex/ [r]
& a
    \ar@2@/_2ex/ [l]
}
\]

\vspace{0.3cm}

\noindent The 1-cell $d$ has two quasi-normal forms which are $a$ and $b$. The 1-cell $c$ is not a quasi-normal form because there is a rewriting step from $c$ to $a$ and $a$ cannot be rewritten into $c$.

\subsubsection{Confluence and convergence}
A $2$-polygraph $\Sigma$ is \emph{confluent} (resp. \emph{locally confluent}) if \linebreak every branching (resp. local branching) $(f,g)$ of $\Sigma$ can be completed by a confluence \linebreak $(f',g') : (t_1(f),t_1(g))~\dfl~v$.
We say that $\Sigma$ is \emph{convergent} (resp. \emph{quasi-convergent}) if it is confluent and it terminates (resp. quasi-terminates).

\subsubsection{Example}
\label{Example:DefinitionSigmaB3+}
The $2$-polygraph $\sigmab=\big\langle s,t \; \big| \; \alpha : sts \dfl tst, \; \beta : tst \dfl sts \big\rangle$
presents the monoid~$\B_3^+$. This polygraph is not terminating but it is quasi-terminating. It has four critical branchings $(\alpha t, s\beta)$, $(\beta s,t \alpha)$, $(\alpha ts, st\alpha)$ and $(\beta st,ts \beta)$. 
These four branchings are confluent as follows
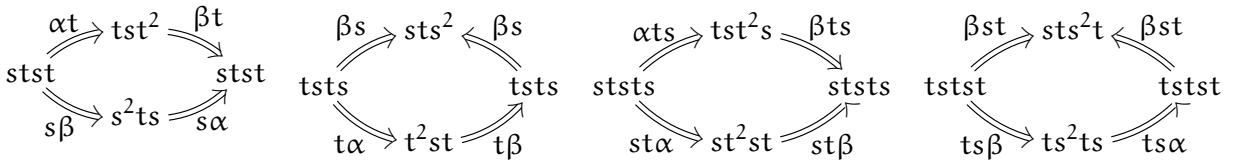
\[
\xymatrix @R=0.15em @C=1.3em {
& tst^2
	\ar@2@/^/ [dr] ^-{\beta t}
\\
stst
	\ar@2@/^/ [ur] ^-{\alpha t}
	\ar@2@/_/ [dr] _-{s\beta}
&& stst
\\
& s^2ts
	\ar@2@/_/ [ur] _-{s\alpha}
}
\;\;
\xymatrix @R=0.6em @C=1.3em {
& sts^2
\\
tsts
	\ar@2@/^/ [ur] ^-{\beta s}
	\ar@2@/_/ [dr] _-{t\alpha}
&& tsts
	\ar@2@/_/ [ul] _-{\beta s}
\\
& t^2st
	\ar@2@/_/ [ur] _-{t\beta}
}
\;\;
\xymatrix @R=0.6em @C=1.3em {
& tst^2s
	\ar@2@/^/ [dr] ^-{\beta ts}
\\
ststs
	\ar@2@/^/ [ur] ^-{\alpha ts}
	\ar@2@/_/ [dr] _-{st\alpha}
&& ststs
\\
& st^2st
	\ar@2@/_/ [ur] _-{st\beta}
}
\;\;
\xymatrix @R=0.6em @C=1.3em {
& sts^2t
\\
tstst
	\ar@2@/^/ [ur] ^-{\beta st}
	\ar@2@/_/ [dr] _-{ts\beta}
&& tstst
	\ar@2@/_/ [ul] _-{\beta st}
\\
& ts^2ts
	\ar@2@/_/ [ur] _-{ts\alpha}
}
\]

\subsubsection{Extended presentations}
Let $\Sigma$ be a $2$-polygraph. A \emph{globular extension} of the $(2,1)$-category~$\tck{\Sigma}_2$ is a set $\Gamma$ together with two maps $s_2,t_2 : \Gamma \fl \tck{\Sigma}_2$ satisfying the \emph{globular relations} $s_1s_2=s_1t_2$ and $t_1s_2=t_1t_2$. Two $2$-cells $f$ and $g$ in $\tck{\Sigma}_2$ are \emph{equal with respect to~$\Gamma$}, and we denote $f \equiv_\Gamma g$, if $f$ and $g$ are equal in the quotient $2$-category $\tck{\Sigma}_2/\Gamma$ of the $2$-category $\tck{\Sigma}_2$ by the congruence on $2$-cells generated by~$\Gamma$.

Relations between rewriting sequences can be described using the notion of extended presentation. Recall from \cite{GuiraudMalbos12advances} that a \emph{$(3,1)$-polygraph} is a pair $(\Sigma,\Sigma_3)$ made of a $2$-polygraph~$\Sigma$ and a globular extension~$\Sigma_3$ of the free $(2,1)$-category $\tck{\Sigma}_2$, that is a set together with two maps \linebreak $s_2,t_2 : \Sigma_3 \fl \tck{\Sigma}_2$ satisfying the \emph{globular relations} $s_1s_2=s_1t_2$ and $t_1s_2=t_1t_2$. We will denote by~$\tck{\Sigma}_3$ the free $(3,1)$-category generated by the $(3,1)$-polygraph $(\Sigma,\Sigma_3)$. An \emph{extended presentation} of a category $\C$ is a $(3,1)$-polygraph whose underlying $2$-polygraph is a presentation of $\C$.

\subsubsection{Coherent presentations}
A \emph{coherent presentation} of a category $\C$ is an extended presentation~$(\Sigma,\Sigma_3)$, such that the globular extension $\Sigma_3$ is a homotopy basis of the $(2,1)$-category $\tck{\Sigma}_2$. That is, for every $2$-sphere $(f,g)$ of $\tck{\Sigma}_2$, there exists a $3$-cell from $f$ to $g$ in the free $(3,1)$-category generated by the $(3,1)$-polygraph~$(\Sigma_2,\Sigma_3)$.

\subsection{Rewriting loops}

In this part, $\Sigma$ denotes a $2$-polygraph.
 
\subsubsection{Equivalent loops}
A \emph{$2$-loop} in the $2$-category $\Sigma_2^\ast$ is a $2$-cell $f$ of $\Sigma_2^\ast$ such that $s_1(f)=t_1(f)$.
Two $2$-loops $f$ and $g$ in $\Sigma_2^\ast$ are \emph{equivalent} if there exist a decomposition $f=f_1\star_1 \ldots \star_1 f_p$, where $f_i$ is a rewriting step of $\Sigma$ for any $1 \leqslant i \leqslant p$, and a circular permutation $\sigma$ such that $g=f_{\sigma(1)} \star_1 \ldots \star_1 f_{\sigma(p)}$ .
This defines an equivalence relation on $2$-cells of $\Sigma_2^\ast$. We will denote by $\Lr(f)$ the equivalence class of a $2$-loop $f$ in $\Sigma_2^\ast$ for this relation.

\begin{lemma}
\label{circulaire}
For any equivalent $2$-loops $f$ and $g$ in $\Sigma_2^\ast$, there exist $2$-cells $h$ and $k$ of $\Sigma_2^\top$  such that $f=h \star_1 g \star_1 k$.
\end{lemma}
\begin{proof}
Let us decompose $f$ into a sequence  $f=f_1\star_1 \ldots \star_1 f_p$ of rewriting steps and let $\sigma$ be a circular permutation such that $g=f_{\sigma(1)} \star_1 \ldots \star_1 f_{\sigma(p)}$. Let $i$ be the integer such that $\sigma(i)=1$. Let $k$ be the $2$-cell $f_{\sigma(1)} \star_1 \ldots \star_1 f_{\sigma(i-1)}$. Let $h=k^-$ be the inverse of $k$ for the $1$-composition. Then, we have $f=h \star_1 g \star_1 k$.
\end{proof}

\subsubsection{Minimal and elementary loops}
We say that a $2$-loop $f$ in $\Sigma_2^\ast$ is
\begin{enumerate}[{\bf i)}]
\item \emph{minimal with respect to $1$-composition}, if any decomposition $f=g\star_1 h \star_1 k$ in $\Sigma_2^\ast$ with $h$ a $2$-loop implies that~$h$ is either an identity or equal to $f$,
\item \emph{minimal by context}, if there is no decomposition $f=ugv$, where $u$ and $v$ are nonidentity $1$-cells in~$\Sigma_1^\ast$ and $g$ is a loop in~$\Sigma_2^\ast$.
\end{enumerate}

A $2$-loop $f$ in $\Sigma_2^\ast$ is \emph{elementary} if it is minimal both with respect to $1$-composition and by context. As an immediate consequence of these definitions, any $2$-loop $f$ minimal for $1$-composition can be written $f=ugv$, where $g$ is an elementary loop and $u$, $v$ are $1$-cells in $\Sigma_1^\ast$.

\begin{lemma}
\label{décomp}
Let $f$ be a nonidentity $2$-loop in $\Sigma_2^\ast$. Then, there exists a decomposition $f=f_1 \star_1 f' \star_1 f_2$ in $\Sigma_2^\ast$, where $f'$ is a $2$-loop minimal with respect to $1$-composition and $f_1$, $f_2$ are $2$-cells such that $f_1 \star_1 f_2$ is a $2$-loop.
\end{lemma}
\begin{proof}
Let $f$ be a nonidentity $2$-loop in $\Sigma_2^\ast$.
The proof is by induction on the support~$\support(f)$. If the $2$-loop $f$ is minimal for $1$-composition, we can write $f=1_{s_1(f)} \star_1 f \star_1 1_{s_1(f)}$. If~$f$ is not minimal for $1$-composition, there exists a decomposition $f=g\star_1 h \star_1 k$, where $h$ is a $2$-loop that is neither an identity nor equal to $f$. Hence, $\support(h)$ is strictly included in $\support(f)$ that proves the decomposition. 
\end{proof}

\subsubsection{Globular extensions of loops}
\label{Subsubsection:GlobularExtensionsLoops}
We will denote by $\mathcal{E}(\Sigma)$ the set of equivalence classes of elementary $2$-loops of $\Sigma_2^\ast$.
A \emph{loop extension} of $\Sigma$ is a globular extension of the $(2,1)$-category $\tck{\Sigma}_2$ made of a family of $3$-cells $A_\alpha : \alpha \tfl 1_{s_1(\alpha)}$ indexed by exactly one $\alpha$ for each equivalence class in $\mathcal{E}(\Sigma)$.

\begin{lemma}
\label{Lemma:decompositionElementaryLoops}
Let $\Lr(\Sigma)$ be a loop extension  of $\Sigma$.  For any $2$-loop $f$ in $\Sigma_2^\ast$, there exists a $3$-cell from $f$ to~$1_{s_1(f)}$ in the free $(3,1)$-category $\tck{\Lr(\Sigma)}$ generated by the $(3,1)$-polygraph $(\Sigma,\Lr(\Sigma))$.
\end{lemma}
\begin{proof}
Let us fix a loop extension $\Lr(\Sigma)$. Let $f$ be $2$-loop in $\Sigma_2^\ast$. We proceed by induction on the support~$\support(f)$.

\noindent {\bf Step 1.} Suppose that $f$ is elementary. By definition of $\Lr(\Sigma)$, the equivalence class $\Lr(f)$ contains an elementary $2$-loop $e$ such that $\Lr(\Sigma)$ contains a $3$-cell $A_e$ from $e$ to~$1_{s_1(e)}$. The $2$-loop $e$ being equivalent to $f$, by Lemma \ref{circulaire} there exist  two $2$-cells $h$ and $k$ of $\Sigma_2^\top$  such that $f=h \star_1 e \star_1 k$. Thus, the $3$-cell $h \star_1 A_e \star_1 k$ in $\Lr(\Sigma)^\top$ goes from $f$ to $h\star_1 k$. By construction the $2$-cell $h\star_1 k$ is equal $1_{s_1(f)}$. In this way we construct a $3$-cell in $\Lr(\Sigma)^\top$ from $f$ to $1_{s_1(f)}$.

\noindent {\bf Step 2.}
Suppose that $f$ is minimal with respect to $1$-composition. Then, there is a decomposition \linebreak $f=ugv$, where $u$ and $v$ are $1$-cells in $\Sigma_1^\ast$ and $g$ is an elementary $2$-loop in $\Sigma_2^\ast$. By Step 1, there exists a $3$-cell $A_g$ from $g$ to~$1_{s_1(g)}$ in $\Lr(\Sigma)^\top$. Thus $uA_gv$ is a $3$-cell in $\Lr(\Sigma)^\top$ from $f$ to $1_{s_1(f)}$.

\noindent {\bf Step 3.} 
Suppose that $f$ is a nonidentity $2$-loop. By Lemma \ref{décomp}, the $2$-loop $f$ can be written as $f_1 \star_1 f' \star_1 f_2$ where $f'$ is a $2$-loop minimal for $1$-composition and $f_1$ and $f_2$ are $2$-cells such that $f_1 \star_1 f_2$ is a $2$-loop. By Step 2, there exists a $3$-cell $A_{f'}$ in $\tck{\Lr(\Sigma)}$ from $f'$ to $1_{s_1(f')}$. Hence, the $1$-composite $f_1\star_1 A_{f'} \star_1 f_2$ is a $3$-cell from $f$ to $f_1 \star_1 f_2$ in $\Lr(\Sigma)^\top$. The support of $f_1\star_1 f_2$ being strictly included in the support of $f$, this proves the lemma by induction on the support of $f$.
\end{proof}

\subsection{Labelled polygraphs}

\subsubsection{Labelled $2$-polygraphs}
A \emph{well-founded labelled $2$-polygraph} is a data $(\Sigma,W,\prec,\psi)$ made of a $2$-polygraph $\Sigma$, a set $W$, a well-founded order $\prec$ on $W$ and a map $\psi:\etapes{\Sigma} \fll W$. The map $\psi$ is called a \emph{well-founded labelling} of $\Sigma$ and associates to a rewriting step $f$ a \emph{label} $\psi(f)$.

Given a rewriting sequence $f=f_1\res \ldots \res f_k$, we denote by $L^W(f) = \{\psi(f_1),\ldots, \psi(f_k)\}$ the set of labels of rewriting steps in $f$. Note that two distinct rewriting sequences $f$ and $g$ can correspond to a same $2$-cell in the free $2$-category $\Sigma_2^\ast$ despite $L^W(f)$ and $L^W(g)$ being distinct.

\subsubsection{Labelling to the quasi-normal form}
Consider a quasi-convergent $2$-polygraph $\Sigma$. By quasi-termination, any $1$-cell $u$ admits a quasi-normal form, not unique in general. For every $1$-cell $u$ in $\Sigma_1^\ast$, let us fix a quasi-normal form $\widetilde{u}$.
Note that by confluence hypothesis, any two congruent $1$-cells of $\Sigma_1^\ast$ have the same quasi-normal form.
This defines a \emph{quasi-normal form map} $s : \Sigma_1^\ast \fl \Sigma_1^\ast$ sending a \linebreak $1$-cell~$u$ on~$\widetilde{u}$.
The \emph{labelling to the quasi-normal form}, labelling QNF for short, associates to the map $s$ the labelling $\labQNF : \etapes{\Sigma} \fll \mathbb{N}$ defined by
\[
\labQNF(f)=d(t_1(f),\widetilde{t_1(f)}),
\]
for any rewriting step $f$ of $\Sigma$.

\subsubsection{Lexicographic maximum measure, {\cite[Definition 3.1]{vanOostrom94}}}
Let $(\Sigma,W,\prec,\psi)$ be a well-founded labelled $2$-polygraph.
Let $w=w_1\ldots w_n$ and $w'=w'_1\ldots w'_m$ be $1$-cells in the free monoid $W^\ast$ with $w_i$ and $w'_j$ in $W$. We denote by $w^{(w')}$ the $1$-cell $\overline{w}_1\ldots\overline{w}_n$
such that for every $0 \leqslant k \leqslant n$, the $1$-cell $\overline{w}_k$ is defined by
\[
\overline{w}_k = 
\begin{cases}
1 & \text{if $w_k \prec w'_j$ for some $1\leq j \leq m$,}\\
w_k & \text{otherwise.}
\end{cases}
\]

Following {\cite[Definition 3.1]{vanOostrom94}}, we consider the measure $|\cdot |$ from the free monoid $W^\ast$ to the set of multisets over $W$ and defined as follows:
\begin{enumerate}[{\bf i)}]
\item for every $i$ in $W$, the multiset $|i|$ is the singleton $\{i\}$,
\item for every $i$ in $W$ and every $1$-cell $w$ in $W^\ast$, we have  $|iw|=|i| \cup |w^{(i)}|$.
\end{enumerate}
The measure $|\cdot |$ is extended to the set of finite rewriting sequences of $\Sigma$ by setting, for every rewriting sequence $f_1\res\ldots \res f_n$, with $f_i$ labelled by $k_i$ for all $i$, 
\[
|f_1\res\ldots \res f_n|=|k_1\ldots k_n|,
\]
were $k_1\ldots k_n$ is a product in the monoid $W^\ast$. Finally, the measure $|\cdot |$ is extended to the set of finite branchings $(f,g)$ of $\Sigma$, by setting
\[
|(f,g)|=|f| \cup |g|.
\]

Recall from {\cite[Lemma 3.2]{vanOostrom94}}, that for every $1$-cells $w_1$, $w_2$ in $W^\ast$, we have $|w_1w_2|=|w_1| \cup |w_2^{(w_1)}|$.
As a consequence, for any rewriting sequences $f$ and $g$ of $\Sigma$ the following relation holds
\[
|f\res g|=|f| \cup |g^{(f)}|,
\]
where $|g^{(f)}|$ is defined by
\[
|g^{(f)}|=|k_1\ldots k_m^{(l_1\ldots l_n)}|,
\]
with $f=f_1\res\ldots \res f_n$ and $g=g_1\res\ldots \res g_m$ and $f_i$ labelled by $l_i$ and $g_j$ labelled by $k_j$.

\subsection{Decreasing two-dimensional polygraphs}

Let us recall in the context of $2$-polygraph the notion of decreasingness from {\cite[Definition 3.3]{vanOostrom94}}.

\subsubsection{Decreasing $2$-polygraph} 
Let $(\Sigma,\psi)$ be a well-founded labelled $2$-polygraph. A local branching $(f,g)$ of $\Sigma$ is \emph{decreasing} (resp. \emph{strictly decreasing}) if there is a confluence diagram of the following form 
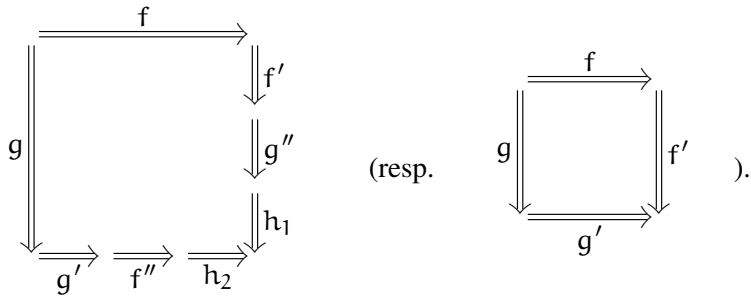
\[
\xymatrix @R=2em @C=2em {
	\ar@2 [rrr] ^-{f}
	\ar@2 [ddd] _-{g}
&&&
	\ar@2 [d] ^-{f'}
\\
&&&
	\ar@2 [d] ^-{g''}
\\
&&&
	\ar@2 [d] ^-{h_1}
\\
	\ar@2 [r] _-{g'}
&
	\ar@2 [r] _-{f''}
&
	\ar@2 [r] _-{h_2}
&
}
\qquad
\raisebox{-1.8cm}{
\text{(resp. }
}
\quad
\raisebox{-0.5cm}{
\xymatrix @R=4.2em @C=4.2em {
	\ar@2 [r] ^-{f}
	\ar@2 [d] _-{g}
&
	\ar@2 [d] ^-{f'}
\\
	\ar@2 [r] _-{g'}
&
}}
\quad 
\raisebox{-1.8cm}{
\text{). }
}
\]
and such that the following properties hold 
\begin{enumerate}[{\bf i)}]
\item $k \prec \psi(f)$, for all $k$ in $L^W(f')$,
\item $k \prec \psi(g)$, for all $k$ in $L^W(g')$,
\item $f''$ is an identity or a rewriting step labelled by $\psi(f)$,
\item $g''$ is an identity or a rewriting step labelled by $\psi(g)$,
\item $k \prec \psi(f)$ or $k \prec \psi(g)$, for all $k$ in $L^W(h_1)\cup L^W(h_2)$.
\end{enumerate}
Such a diagram is then called a \emph{decreasing confluence diagram} (resp. \emph{strictly decreasing confluence diagram}) of the branching $(f,g)$. 

A $2$-polygraph $\Sigma$ is \emph{decreasing} (resp. \emph{strictly decreasing}) if there exists a well-founded labelling $(W,\prec,\psi)$ of $\Sigma$ making all its local branching decreasing (resp. strictly decreasing).

As in the case of abstract rewriting systems, {\cite[Corollary 3.9.]{vanOostrom94}}, we prove that any decreasing $2$-polygraph is confluent.

\subsubsection{Strictly decreasing branching}
We extend the notion of strict decreasingness on local branchings to branchings as follows. A branching $(f,g)$ is \emph{strictly decreasing} is there is a confluence diagram $(f\cdot f',g\cdot g')$ such that the two following properties hold
\begin{enumerate}[{\bf i)}]
\item for each $k'$ in $L^W(f')$, we have $k'\prec k$ for any $k$ in $L^W(f)$,
\item for each $l'$ in $L^W(g')$, we have $l'\prec l$ for any $l$ in $L^W(g)$.
\end{enumerate}

\subsubsection{Decreasingness from quasi-termination}
\label{Subsubsection:DecreasingnessFromQuasiTermination}
Any quasi-convergent $2$-polygraph $\Sigma$ is strictly decreasing with respect to any quasi-normal form labelling $\labQNF$.
Indeed, for any local branching \linebreak $u \dfl (v,w)$ there exists a quasi-normal form $\widetilde{u}$ and a confluence $(f',g') : (v,w) \dfl \widetilde{u}$.
The rewriting sequences $f'$ and $g'$ can be chosen of minimal length, thus making this confluence diagram strictly decreasing with respect the labelling $\labQNF$.

\subsubsection{Decreasingness of Peiffer branchings}
For a Peiffer branching $(f v, u g) : uv \dfl (u'v,uv')$ of a $2$-polygraph $\Sigma$, the confluence $(u'g,fv') : (u'v,uv') \dfl u'v'$ is called the \emph{Peiffer confluence} of the branching $(fv,ug)$.
In a decreasing $2$-polygraph $(\Sigma,\psi)$ every Peiffer branching can be completed into a decreasing confluence diagram. However, the confluence diagram obtained with the Peiffer confluence is not always decreasing as in the case of the following example.

\subsubsection{Example}
\label{Example:a=b}
As shown in \ref{Subsubsection:DecreasingnessFromQuasiTermination}, a labelling QNF makes every Peiffer branching decreasing. But, it does not necessarily makes the Peiffer confluences decreasing. In particular, it is not the case when the source $uv$ of the Peiffer confluence is already the chosen quasi-normal form. For instance, consider the quasi-convergent $2$-polygraph $\Sigma=\langle a,b \; \big| \; \alpha : a \dfl b,\,\beta : b \dfl a\rangle$. For each 1-cel $u$ of~$\Sigma_1^\ast$, we set~$\widetilde{u}=a^{\ell(u)}$ as a quasi-normal form.
Let us now consider the following Peiffer diagram:
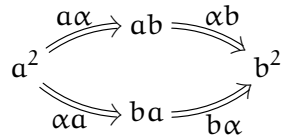
\[
\xymatrix @R=0.15em @C=2.5em {
& ab
	\ar@2@/^/ [dr] ^{\alpha b}
\\
a^2
	\ar@2@/^/ [ur] ^-{a\alpha}
	\ar@2@/_/ [dr] _-{\alpha a}
&& b^2
\\
& ba
	\ar@2@/_/ [ur] _-{b\alpha}
}
\]
This Peiffer diagram is not decreasing with respect to $\labQNF$. 
Indeed, we have $\labQNF(\alpha a)=\labQNF(a\alpha)=1$ and $\labQNF(\alpha b) = \labQNF(b\alpha) =2$. However, this Peiffer branching is decreasing by using the following confluence $ (a\beta,\beta a) : (ab,ba) \dfl a^2$, since $\labQNF(a\beta)=\labQNF(\beta a)=0$.

\subsubsection{Peiffer decreasingness}
A decreasing (resp. strictly decreasing) $2$-polygraph $(\Sigma,\psi)$ is \emph{Peiffer decreasing} with respect to a globular extension $\Gamma$ of the $(2,1)$-category $\tck{\Sigma}_2$ if, for any Peiffer branching $(fv,ug): uv \dfl (u'v,uv')$, there exists a decreasing (resp. strictly decreasing) confluence diagram $(fv\res f',ug \res g')$:
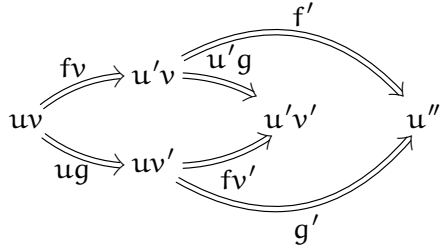
\[
\xymatrix @R=0.15em @C=2.5em {
& u'v
	\ar@2@/^/ [dr] ^-{u'g}
	\ar@2@/^5ex/ [drr] ^-{f'}
\\
uv
	\ar@2@/^/ [ur] ^-{fv}
	\ar@2@/_/ [dr] _-{ug}
&& u'v' & u''
\\
& uv'
	\ar@2@/_/ [ur] _-{fv'}
    \ar@2@/_5ex/ [urr] _-{g'}
}
\]
such that
$u'g \star_1 (fv')^- \equiv_\Gamma 
f' \star_1 (g')^-$.

\subsubsection{Example}
\label{Remark:DecreasingInTwoSteps}
Any $2$-polygraph $\Sigma$ such that any non trivial local branching $(f,g)$ is confluent using two rewriting steps $f':t_1(f)\dfl v$ and $g':t_1(g)\dfl v$ is Peiffer decreasing.
Indeed, a labelling such that all rewriting steps have the same label makes any local branching $(f,g)$ decreasing. Moreover, with such a labelling, any Peiffer confluence is decreasing.
In particular, the $2$-polygraph $\sigmab$ is decreasing for a singleton labelling.

\subsubsection{Compatibility with contexts}
Let $(\Sigma,\psi)$ be a well-founded labelled $2$-polygraph. The labelling~$\psi$ is \emph{compatible with contexts} if for any decreasing (resp. strictly decreasing) confluence diagram \linebreak $(f\cdot f', g\cdot g')$, where $(f,g)$ is a local branching, and for any composable $1$-cells $u_1$ and $u_2$ in $\Sigma_1^\ast$, the following confluence diagram is decreasing (resp. strictly decreasing):
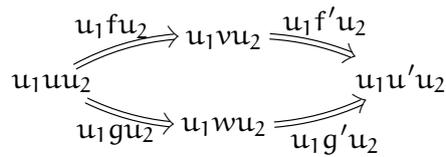
\[
\xymatrix @R=0.15em @C=2.5em {
& u_1vu_2
	\ar@2@/^/ [dr] ^-{u_1f'u_2}
\\
u_1uu_2
	\ar@2@/^/ [ur] ^-{u_1fu_2}
	\ar@2@/_/ [dr] _-{u_1gu_2}
&& u_1u'u_2
\\
& u_1wu_2
	\ar@2@/_/ [ur] _-{u_1g'u_2}
}
\]

Note that a labelling QNF is not compatible with contexts in general.

\subsubsection{$\star_0$-compatibility}
A well-founded labelling $(W,\psi,\prec)$ is \emph{$\star_0$-compatible} if for any rewriting steps~$f$ and $g$ such that $\psi(f) \prec \psi(g)$, we have $\psi(u_1fu_2) \prec \psi(u_1gu_2)$ for any composable $1$-cells $u_1$ and~$u_2$ in~$\Sigma_1^\ast$.
Note that the $\star_0$-compatibility does not implies the compatibility with contexts. Indeed, if $(f,g)$ is a local branching that can be completed into a diagram
\[
\xymatrix @R=0.15em @C=2.5em {
& b
    \ar@2@/^/ [dr] ^-{f'}
\\
a
	\ar@2@/^/ [ur] ^-{f}
	\ar@2@/_/ [dr] _-{g}
&&
d
\\
& c
  \ar@2@/_/ [ur] _-{g'}
}
\]
where $f'$ and $g'$ are rewriting steps such that $\psi (f)=\psi (f')$ and $\psi (g)=\psi (g')$, then the confluence diagram is decreasing. Even, if the labelling $(W,\psi,\prec)$ is $\star_0$-compatible,  we do not necessarily have $\psi (ufv)=\psi (uf'v)$ and $\psi (ugv)=\psi (ug'v)$ for any 1-cells $u$ and $v$. Thus, the following diagram is not decreasing in general:
\[
\xymatrix @R=0.15em @C=2.5em {
& ubv
    \ar@2@/^/ [dr] ^-{uf'v}
\\
uav
	\ar@2@/^/ [ur] ^-{ufv}
	\ar@2@/_/ [dr] _-{ugv}
&&
udv
\\
& ucv
  \ar@2@/_/ [ur] _-{ug'v}
}
\]

If $\psi$ is a $\star_0$-compatible labelling, for any strictly decreasing diagram $(f\cdot f',g\cdot g')$, where $(f,g)$ is a local branching, we have 
$\labQNF(u_1f'u_2) < \labQNF(u_1fu_2)$ and $\labQNF(u_1g'u_2) < \labQNF(u_1gu_2)$ for every composable $1$-cells $u_1$ and $u_2$.
As a consequence, any $\star_0$-compatible labelling on a strictly decreasing $2$-polygraph is compatible with contexts.

\subsubsection{Example}
Consider the $2$-polygraph $\Sigma$ defined in \ref{Example:a=b}. The labelling QNF defined using the quasi-normal forms of the form $\widetilde{u}=a^{\ell(u)}$ is compatible with contexts. This is a consequence of the following equality
\[
\labQNF(u_1fu_2)=d(u_1,a^{\ell(u_1)})+\labQNF(f)+d(u_2,a^{\ell(u_2)})
\]
for any rewriting step $f$ and $1$-cells $u_1$ and $u_2$.

If we consider an other labelling QNF of the $2$-polygraph $\Sigma$  associated to quasi-normal forms of the form $\widetilde{u}=a^{\ell(u)}$ for any $1$-cell $u$ such that $\ell(u)\neq 3$ and $\widetilde{u}=b^3$ for any $1$-cell $u$ such that $\ell(u)=3$.
Then the confluence diagram
$(a\alpha \cdot a\beta , \alpha a \cdot \beta a)$ is decreasing with $\labQNF(a\alpha)=\labQNF(\alpha a)= 1$ and $\labQNF(a\beta)=\labQNF(\beta a)=0$. However, the confluence diagram
$(ba\alpha \cdot ba\beta, b\alpha a \cdot b\beta a)$ is not decreasing with $\labQNF(ba\alpha)=\labQNF(b\alpha a)= 1$ and $\labQNF(ba\beta)=\labQNF(b\beta a)=2$. As a consequence this labelling QNF is not compatible with contexts.

\subsubsection{Example}
\label{Example:QNFSigmaB3+}
Consider the $2$-polygraph $\sigmab$ given in~\ref{Example:DefinitionSigmaB3+}. We define a QNF labelling $\labQNF$ on~$\sigmab$ by associating to each $1$-cell $u$ of $\sigmab_1^\ast$ the quasi-normal form $\widetilde{u}$ defined as follows. 
Setting $N_u=\max\{n\;|\;\text{$u=(sts)^nv$ \, holds in $\B_3^+$}\}$, we define $\widetilde{u}= (sts)^{N_u}v$. The maximality of $N_u$ ensures the unicity of such a quasi-normal form. Indeed, let us consider the following convergent presentation of the monoid $\B_3^+$:
\[
\Upsilon=\big\langle s,t,a \; \big| \; sts \dfl a, \; tst \dfl a, sa \dfl at, ta \dfl as \big\rangle.
\]
Suppose that a $1$-cell $u$ of $\sigmab_1^\ast$ has two distinct quasi-normal forms $(sts)^{N_u}v$ and $(sts)^{N_u}w$. Those two 1-cells have respectively $a^{N_u}v$ and $a^{N_u}w$ as normal forms with respect to $\Upsilon_2$. Indeed, there is no occurrence of~$a$ in $v$ and $w$, and the $1$-cells $sts$ and $tst$ cannot divide $v$ and $w$ by maximality of $N_u$. By unicity of the normal forms in a convergent $2$-polygraph, the 1-cells $a^{N_u}v$ and $a^{N_u}w$ are not equal in the monoid $\B_3^+$, hence they are not the normal forms of a same $1$-cell in $\Upsilon_1^\ast$. Thus, the 1-cells $(sts)^{N_u}v$ and $(sts)^{N_u}w$  are not the quasi-normal forms of a same 1-cell in $\sigmab_1^\ast$, which contradicts our assumption.

The labelling defined in this way is $\star_0$-compatible. Indeed, for any rewriting steps $f$ and $g$ of $\sigmab$ such that $\labQNF(g) < \labQNF(f)$ and for any composable $1$-cells $u_1$ and $u_2$, we have \linebreak $\labQNF(u_1fu_2) < \labQNF(u_1gu_2)$.
Hence, the labelling $\labQNF$ is compatible with contexts.

\subsubsection{Multiset order}
\label{Subsubsection:MultisetOrder}
Given a well-founded set of labels $(W,\prec)$, we consider the partial order $\precmult$ on the multisets over $W$ defined in \cite{vanOostrom94, DershowitzManna79} as follows. For any multisets $M$ and $N$ over $W$, we set $M\precmult N$ if there exist multisets $X$, $Y$ and $Z$ such that:
\begin{enumerate}[{\bf i)}]
\item $M=Z \cup X$, $N=Z \cup Y$ and $Y$ is not empty,
\item for every $i$ in $W$ such that $X(i)\neq 0$, there exists $j$ in $W$ such that $Y(j)\neq 0$ and $i\prec j$.
\end{enumerate}
The order $\precmult$ is well-founded because $\prec$ is. We call $\precmultE$ the symmetric closure of $\precmult$.

Let us mention a particular case of {\cite[Lemma 3.6.]{vanOostrom94}}, that will be used in the proof of our main result.

\begin{lemma}
\label{Lemma:vanOostrom}
Let $\Sigma$ be a decreasing $2$-polygraph.
For every diagram in $\Sigma_2^\ast$ of the following form
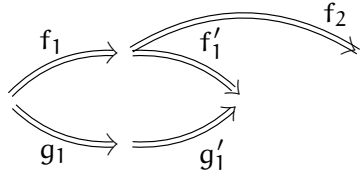
\[
\xymatrix @R=1em @C=3.5em {
& 
	\ar@2@/^/ [dr] ^(0.6){f_1'}
	\ar@2@/^3ex/ [rr] ^(0.85){f_2}
	&&
\\
	\ar@2@/^/ [ur] ^-{f_1}
	\ar@2@/_/ [dr] _-{g_1}
&&&
\\
&
	\ar@2@/_/ [ur] _(0.6){g_1'}
}
\]
where $f_1$ is a non empty rewriting sequences, $f_2$ and $g_1$ are rewriting sequence and  the confluence diagram $(f_1\cdot f'_1, g_1\cdot g'_1)$ is strictly decreasing, the inequality $|(f'_1,f_2)|\precmult |(g_1,f_1 \res f_2)|$ holds.
\end{lemma}

\begin{proposition}
\label{strictglobdecr}
Let $(\Sigma,\psi)$ be a well-founded labelled $2$-polygraph.
Then $\Sigma$ is strictly decreasing if and only if any branching of $\Sigma$ is strictly decreasing.
\end{proposition}
\begin{proof}
One implication is obvious. Let us assume that $\Sigma$ is strictly decreasing and let $(f,g)$ be a branching of $\Sigma$. We prove by induction on $|(f,g)|$ that $(f,g)$ is strictly decreasing. If $f$ or $g$ is an empty rewriting sequence, the strict decreasingness of $(f,g)$ is trivial. Else, we can write
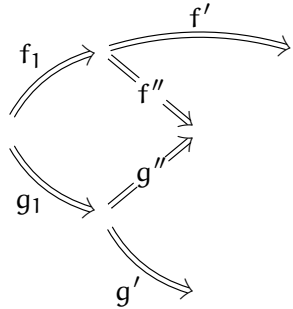
\[
\xymatrix @R=1em @C=2.8em{
& \strut
	\ar@2@/^/ [rr] ^-{f'} 
	\ar@2 [dr] |-{f''}
& \strut
& \strut
\\
\strut 
	\ar@2@/^/ [ur] ^-{f_1}
	\ar@2@/_/ [dr] _-{g_1}
&& \strut
&& \strut
\\
& \strut
	\ar@2 [ur] |-{g''}
	\ar@2@/_/ [dr] _-{g'}
\\
&& \strut
&& \strut
}
\]
such that the confluence diagram $(f_1 \res f'',g_1\res g'')$ is strictly decreasing. By Lemma \ref{Lemma:vanOostrom}, we have $|(f',f'')| \precmult |(f_1,g)|$. Thus, we have $|(f',f'')| \precmult |(f,g)|$ and we can use the induction hypothesis to construct a strictly decreasing confluence diagram $(f' \res k_1,f''\res k_2)$. By using again Lemma \ref{Lemma:vanOostrom}, we have $|(g'' \res k_2,g')| \precmult |(f,g)|$. Thus, by applying again the induction hypothesis, we have a diagram
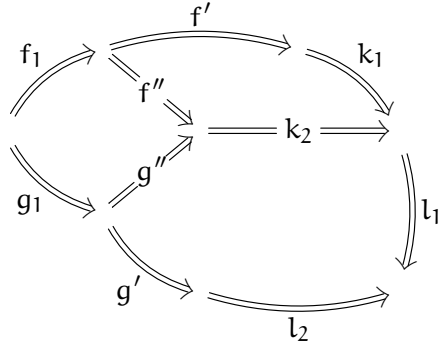
\[
\xymatrix @R=1em @C=2.8em{
& \strut
	\ar@2@/^/ [rr] ^-{f'} 
	\ar@2 [dr] |-{f''}
& \strut
& \strut
	\ar@2@/^/ [dr] ^-{k_1}
\\
\strut 
	\ar@2@/^/ [ur] ^-{f_1}
	\ar@2@/_/ [dr] _-{g_1}
&& \strut
	\ar@2 [rr] |-*+{k_2}
&& \strut
	\ar@2@/^/ [dd] ^-{l_1}
\\
& \strut
	\ar@2 [ur] |-{g''}
	\ar@2@/_/ [dr] _-{g'}
\\
&& \strut
	\ar@2@/_/ [rr] _-{l_2}
&& \strut
}
\]
where the diagram $(f \res k_1 \res l_1,g \res l_2)$ is strictly decreasing.
\end{proof}

\section{Coherence by decreasingness}

In this section, we extend to decreasing $2$-polygraphs the notion of Squier's completion known for convergent $2$-polygraphs. We give sufficient conditions on the labelling of a decreasing $2$-polygraph making the Squier's decreasing completion a coherent presentation. In particular, we show how to extend a quasi-convergent $2$-polygraph into a coherent presentation.

\subsection{Squier's decreasing completion}
Squier's completion provides a way to extend a convergent $2$-polygraph into a coherent presentation, see \cite{Squier94, GuiraudMalbos16}.

\subsubsection{Squier's completion}
A \emph{family of generating confluences} of a $2$-polygraph $\Sigma$ is a globular extension of the $(2,1)$-category $\tck{\Sigma}_2$ that contains exactly one $3$-cell of the following form
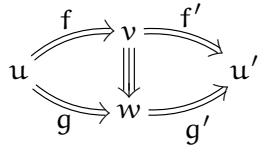
\[
\xymatrix @R=0.15em @C=2.5em{
& v
    \ar@2@/^/ [dr] ^-{f'}
\\
u
	\ar@2@/^/ [ur] ^-{f}
	\ar@2@/_/ [dr] _-{g}
&&
u'
\\
& w
  \ar@2@/_/ [ur] _-{g'}
\ar@3 "1,2"!<0pt,15pt>;"3,2"!<0pt,-15pt> ^-{}
}
\]
for each critical branching $(f,g)$ of $\Sigma$. If $\Sigma$ is confluent, it always admits such a family
A \emph{Squier's completion} of a convergent $2$-polygraph $\Sigma$ is a $(3,1)$-polygraph that extends $\Sigma$ by a chosen family of generating confluences. Any Squier's completion of a convergent $2$-polygraph $\Sigma$ is a coherent presentation of the category presented by $\Sigma$, \cite{Squier94}, see also \cite{GuiraudMalbos16}.

\subsubsection{Generating decreasing confluences}
Let $(\Sigma,\psi)$ be a decreasing $2$-polygraph.
A \emph{family of generating decreasing confluences of $\Sigma$ with respect to $\psi$} is a globular extension of the $(2,1)$-category $\tck{\Sigma}_2$ that contains, for every critical branching $(f,g): u \dfl (v,w)$ of $\Sigma$, exactly one $3$-cell $D_{f,g}^\psi$ of the following form
\[
\xymatrix @R=0.15em @C=2.5em{
& v
    \ar@2@/^/ [dr] ^-{f'}
\\
u
	\ar@2@/^/ [ur] ^-{f}
	\ar@2@/_/ [dr] _-{g}
&&
u'
\\
& w
  \ar@2@/_/ [ur] _-{g'}
\ar@3 "1,2";"3,2" ^-{D_{f,g}^\psi}
}
\]
and where the confluence diagram $(f\res f', g\res g')$ is decreasing with respect to $\psi$.
Any decreasing $2$-polygraph admits such a family of generating decreasing confluences. Indeed, 
any critical branching is local and thus confluent by decreasingness hypothesis. However, note that such a family is not unique in general.

For a strictly decreasing $2$-polygraph $\Sigma$, we define in the same way a \emph{family of generating strictly decreasing confluences} of $\Sigma$, but where the confluence diagrams are strictly decreasing with respect to~$\psi$.

\subsubsection{Squier's decreasing completion}
Let $(\Sigma,\psi)$ be a decreasing $2$-polygraph. A \emph{Squier's decreasing completion of $\Sigma$ with respect to $\psi$} is a $(3,1)$-polygraph that extends the $2$-polygraph $\Sigma$ by a globular extension 
\[
\Or(\Sigma,\psi) \cup \Lr(\Sigma)
\]
where $\Or(\Sigma,\psi)$ is a chosen family of generating decreasing confluences with respect to $\psi$ and $\Lr(\Sigma)$ is a loop extension of $\Sigma$ defined in \ref{Subsubsection:GlobularExtensionsLoops}. If $(\Sigma,\psi)$ is a strictly decreasing $2$-polygraph, a \emph{strictly decreasing Squier's completion} is a Squier's decreasing completion, whose the generating decreasing confluences are required strict.

\begin{lemma}
\label{Lemma1}
Let $(\Sigma,\psi)$ be a strictly decreasing $2$-polygraph.
Let $\Sr^{sd}(\Sigma,\psi)$ be a strictly decreasing Squier's completion of $\Sigma$. Suppose that $\psi$ is compatible with contexts and that $(\Sigma,\psi)$ is Peiffer decreasing with respect to the extension $\Sr^{sd}(\Sigma,\psi)$.
Then, for any $2$-sphere $(f,g)$ in~$\Sigma_2^\ast$, there exists a $3$-cell from $f$ to $g$ in the $(3,1)$-category $\Sr^{sd}(\Sigma,\psi)^\top$.
\end{lemma}

\begin{proof}
We proceed in two steps.

\noindent {\bf Step 1.} We prove that, for every local branching $(f,g):u\dfl(v,w)$ of~$\Sigma$, there exists a confluence $(f',g'):(v,w)\dfl u'$ of~$\Sigma$ and a $3$-cell $A:f\star_1 f'\tfl g\star_1 g'$ in~$\tck{\Sr^{sd}(\Sigma,\psi)}$ such that the confluence diagram $(f\res f', g\res g')$ is strictly decreasing.

In the case of an aspherical branching, we can choose~$f'$ and~$g'$ to be identity $2$-cells, $A$ to be an identity $3$-cell and the confluence diagram $(f,f)$ is trivially strictly decreasing.

Suppose that $(f,g)$ is a Peiffer branching $(f_1v_1,u_1g_1): u_1v_1 \dfl (u_1'v_1,u_1v_1')
$. By hypothesis, the Peiffer confluence $(f_1v_1\res u'_1g_1, u_1g_1\res f_1v'_1)$ is equivalent to a strictly decreasing confluence diagram \linebreak $(f_1v_1\res f'_1, u_1g_1\res g_1')$. Hence, there exists a $3$-cell $A:f_1v_1\star_1 f'_1\tfl u_1g_1\star_1 g'_1$ in the $(3,1)$-category $\tck{\Sr^{sd}(\Sigma,\psi)}$.

If $(f,g)$ is an overlapping branching, we have $(f,g)=(whw',wkw')$ with~$(h,k)$ a critical branching. We consider the $3$-cell $D_{h,k}^{\psi}:h\star_1 h'\tfl k\star_1 k'$ of~$\Or(\Sigma,\psi)$ corresponding to the strict generating decreasing confluence of the critical branching~$(h,k)$ with respect to the labelling $\psi$, or its inverse.
Let us define the $2$-cells~$f'=wh'w'$ and~$g'=wk'w'$ and the $3$-cell~$A=wD_{h,k}^{\psi} w'$.
The labelling $\psi$ being compatible with contexts, the confluence diagram corresponding to the $3$-cell $A$ is strictly decreasing.

\medskip

\noindent {\bf Step 2.} Let $(f,g)$ be a $2$-sphere in $\Sigma_2^\ast$. This $2$-sphere defines a branching with source $s_1(f)=s_1(g)$.
The $2$-polygraph $\Sigma$ being strictly decreasing, we prove the lemma by well-founded induction on the measure $|(f,g)|$ of the branching $(f,g)$.
If $f$ or $g$ is an identity $2$-cell, say $g=1$, the $2$-cell $f$ is a $2$-loop. By Lemma \ref{Lemma:decompositionElementaryLoops}, there exists a $3$-cell $E : f \tfl 1_{s_1(f)}$ in the $(3,1)$-category $\tck{\Lr(\Sigma)}$.
Else, we have decompositions $f=f_1 \star_1 f_2$ and $g=g_1 \star_1 g_2$ in $\Sigma_2^\ast$ where $(f_1,g_1)$ is a local branching. Note that $f_2$ or $g_2$ can be equal to an identity $2$-cell.
The local branching $(f_1,g_1)$ is confluent by decreasingness. Moreover, by Step 1, there exists a $3$-cell $A: f_1\star_1 f'_1 \tfl g_1\star_1 g'_1$ in the $(3,1)$-category $\tck{\Sr^{sd}(\Sigma,\psi)}$, where the confluence diagram $(f_1\res f_1', g_1\res g_1')$ is strictly decreasing.

The branchings $(f'_1,f_2)$ is confluent by decreasingness. Moreover, the $2$-polygraph $\Sigma$  being strictly decreasing, by Lemma \ref{strictglobdecr}, there exist rewriting sequences $h$ and $k$ as indicated in the following diagram:
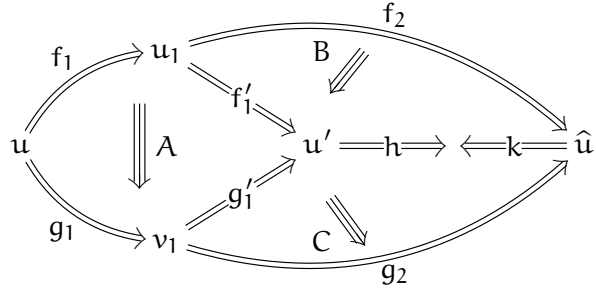
\[
\xymatrix@R=2em@C=3.6em{
& {u_1}
	\ar@2[dr] |-{f'_1}
	\ar@2 @/^5ex/[rrrd] ^-{f_2} _-{}="5"
        \ar@3 []!<-10pt,-20pt>;[dd]!<-10pt,20pt> ^*+{A}
&&
\\
{u}
	\ar@2 @/^2ex/ [ur] ^-{f_1}
	\ar@2 @/_2ex/ [dr] _-{g_1}
&& {u'}
	\ar@2[r] |-{h}
&& {\rep{u}}
    \ar@2[l] |-{k}
\\
& {v_1}
	\ar@2[ru] |-{g'_1}
	\ar@2 @/_5ex/[rrru] _-{g_2} ^-{}="8"
&&
\ar@3 "5"!<-9.55pt,-6.5pt>;"2,3"!<5pt,20pt> _*+{B}
\ar@3 "2,3"!<5pt,-20pt>;"8"!<-9.2pt,4.5pt> _*+{C}
}
\]
such that the confluence diagrams $(f'_1\cdot h, f_2\cdot k)$ is strictly decreasing.

Consider the multiset order $\precmult$ associated to the order $\prec$.
The confluence diagram $(f_1\res f'_1,g_1\res g'_1)$ being strictly decreasing, for any $k$ in $L^W(f'_1)$ and any $l$ in $L^W(g'_1)$, we have $k \prec \psi(f_1)$ and $l \prec \psi(g_1)$. Thus $|f_1 \cdot f'_1|=|f_1|$ and $|g_1 \cdot g'_1|=|g_1|$. This implies the following equality
\[
|(f_1,g_1)| = |(f,g)|.
\]
The confluence diagram $(f'_1\cdot h, f_2\cdot k)$ being strictly decreasing, by the same argument, we have
\[
|(f_1'\res h,f_2 \res k)|= |(f_1',f_2)|.
\]
Moreover, by Lemma \ref{Lemma:vanOostrom}, we have 
$|(f_1',f_2) | \precmult |(f,g_1)|$.
It follows that
\[
|(f_1'\res h,f_2 \res k)| \precmult |(f,g)|.
\]
By induction hypothesis, we deduce that there exists a $3$-cell $B : f_2\star_1 k \tfl f_1'\star_1 h$ in $\tck{\Sr^{sd}(\Sigma,\psi)}$.

Finally, let us prove that there exists a $3$-cell $C: g_1'\star_1 h \tfl g_2 \star_1 k$ in $\tck{\Sr^{sd}(\Sigma,\psi)}$. 
We have 
\[
|(g'_1\cdot h, g_2\cdot k)| = |g'_1| \cup |h^{(g'_1)}| \cup |g_2| \cup |k^{(g_2)}|.
\]
On the other hand, we have
\[
|(f,g)| = |f| \cup |g| = |f| \cup |g_1| \cup |g_2^{(g_1)}|.
\]
Furthermore, there exists a multiset $R$, possibly empty, such that $|g_2| = |g_2^{(g_1)}| \cup R$. 
Hence 
\[
|(g'_1\cdot h, g_2\cdot k)| =   |g_2^{(g_1)}| \cup X
\quad\text{and}\quad
|(f,g)| = |g_2^{(g_1)}| \cup Y.
\]
where $X = |g'_1| \cup |h^{(g'_1)}| \cup R \cup |k^{(g_2)}|$ and 
$Y = |f| \cup |g_1|$.
Moreover, we check that for every $i$ in $W$ such that $X(i)\neq 0$, there exists $j$ in $W$ such that $Y(j)\neq 0$ and $i \prec j$. 
Hence, we have 
\[
|(g'_1\cdot h, g_2\cdot k)| \precmult |(f,g)|.
\]
The existence of the $3$-cell $C$ follows by induction hypothesis.
In this way, we have constructed a $3$-cell in $\tck{\Sr^{sd}(\Sigma,\psi)}$ from $f$ to $g$ obtained by composition of the $3$-cells $A$, $B$ and $C$.
\end{proof}

\subsubsection{Example}
\label{Example:vOSCompletionSigmaB3+}
The $2$-polygraph $\sigmab$ given in~\ref{Example:DefinitionSigmaB3+} is strictly decreasing for the labelling QNF $\labQNF$ defined in~\ref{Example:QNFSigmaB3+}.
It has four confluent critical branchings. Thus, a strictly decreasing Squier's completion of the $2$-polygraph $\sigmab$ is given by the following $3$-cells:
\[
\xymatrix @R=0.6em @C=2.5em {
& tst^2
	\ar@2@/^/ [dr] ^-{\beta t}
\\
stst
	\ar@2@/^/ [ur] ^-{\alpha t}
	\ar@2@/_/ [dr] _-{s\beta}
&& stst
\\
& s^2ts
	\ar@2@/_/ [ur] _-{s\alpha}
\ar@3 "1,2";"3,2" ^-{D_{t\alpha ,\beta s}^{\labQNF}}
}
\qquad
\xymatrix @R=0.6em @C=2.5em {
& sts^2
\\
tsts
	\ar@2@/^/ [ur] ^-{\beta s}
	\ar@2@/_/ [dr] _-{t\alpha}
&& tsts
	\ar@2@/_/ [ul] _-{\beta s}
\\
& t^2st
	\ar@2@/_/ [ur] _-{t\beta}
\ar@3 "1,2";"3,2" ^-{D_{\beta s,t\alpha}^{\labQNF}}
}
\qquad
\xymatrix @R=0.6em @C=2.5em {
& tst^2s
	\ar@2@/^/ [dr] ^-{\beta ts}
\\
ststs
	\ar@2@/^/ [ur] ^-{\alpha ts}
	\ar@2@/_/ [dr] _-{st\alpha}
&& ststs
\\
& st^2st
	\ar@2@/_/ [ur] _-{st\beta}
\ar@3 "1,2";"3,2" ^-{D_{\alpha ts,st\alpha}^{\labQNF}}
}
\]
\[
\xymatrix @R=0.6em @C=2.5em {
& sts^2t
\\
tstst
	\ar@2@/^/ [ur] ^-{\beta st}
	\ar@2@/_/ [dr] _-{ts\beta}
&& tstst
	\ar@2@/_/ [ul] _-{\beta st}
\\
& ts^2ts
	\ar@2@/_/ [ur] _-{ts\alpha}
\ar@3 "1,2";"3,2" ^-{D_{\beta st,ts\beta}^{\labQNF}}
}
\qquad
\xymatrix @R=0.6em @C=2.5em {
&tst
           \ar@2@/^/ [dr] ^-{\beta}
\\
sts
	\ar@2@/^/ [ur] ^-{\alpha}
	\ar@2@/_2ex/ [rr] _-{1_{sts}} ^{}="but"
&& sts
\ar@3 "1,2";"but" ^-{E_{\alpha ,\beta}}
}
\]
where $D_{t\alpha ,\beta s}^{\labQNF}$, $D_{\beta s,t\alpha}^{\labQNF}$, $D_{\alpha ts,st\alpha}^{\labQNF}$ and $D_{\beta st,ts\alpha}^{\labQNF}$ are the generating decreasing confluences and $E_{\alpha\star_1\beta}$ is an elementary $2$-loop of $\Sigma$. Each of these confluences is decreasing because:
\[
\labQNF(\alpha t)=\labQNF(s\beta)=1
\quad\text{and}\quad
\labQNF(\beta t)=\labQNF(s\alpha)=0,
\]
\[
\labQNF(\beta s)=0,\;\labQNF(t\alpha)=2
\quad\text{and}\quad
\labQNF(t\beta)=1, \labQNF(\beta s)=0,
\]
\[
\labQNF(\alpha ts)=\labQNF(st\alpha)=1
\quad\text{and}\quad
\labQNF(\beta ts)=\labQNF(st\beta)=0,
\]
\[
\labQNF(\beta st)=0,\;\labQNF(ts\beta)=2
\quad\text{and}\quad
\labQNF(ts\alpha)=1, \labQNF(\beta st)=0.
\]

\subsection{Coherence by decreasingness}

The following theorem is the main result of this article.

\begin{theorem}
\label{Theorem:MainResult}
Let $(\Sigma,\psi)$ be a strictly decreasing $2$-polygraph.
Let $\Sr^{sd}(\Sigma,\psi)$ be a strictly decreasing Squier's completion of $\Sigma$.
If $\psi$ is compatible with contexts and $(\Sigma,\psi)$ is Peiffer decreasing with respect to the extension $\Sr^{sd}(\Sigma,\psi)$, then $\Sr^{sd}(\Sigma,\psi)$ is a coherent presentation of the category presented by $\Sigma$.
\end{theorem}

\begin{proof}
Let $(f,g)$ be a $2$-sphere of the $(2,1)$-category $\tck{\Sigma}_2$. 
By definition of $\tck{\Sigma}_2$, the $2$-cell $f\star_1 g^-$ can be decomposed into a zigzag
\[
\xymatrix@C=4em @R=2em{
&
&
  \ar@2 [l] _-{f_1}
  \ar@2 [r] ^-{f_2}
&
\qquad\cdots\qquad
&
  \ar@2 [l] _-{f_{k-2}}
  \ar@2 [r] ^-{f_{k-1}}
&
\\
  \ar@2 @/_/ [dr] _-{g_0}
  \ar@2 @/^/ [ur] ^-{f_0}
&&&&&&
  \ar@2 @/^/  [dl] ^-{g_l}
  \ar@2 @/_/  [ul] _-{f_k}
\\
&
&
  \ar@2 [l] ^-{g_1}
  \ar@2 [r] _-{g_2}
&
\qquad\cdots\qquad
&
  \ar@2 [l] ^-{g_{l-2}}
  \ar@2 [r] _-{g_{l-1}}
&
}
\]
where the $2$-cells $f_0,\ldots, f_k$ and $g_0,\ldots,g_l$ are $2$-cells of the $2$-category $\Sigma_2^\ast$. Note that some of those $2$-cells can be identities.
By confluence of the $2$-polygraph $\Sigma$, there exist families of $2$-spheres of $\Sigma_2^\ast$ 
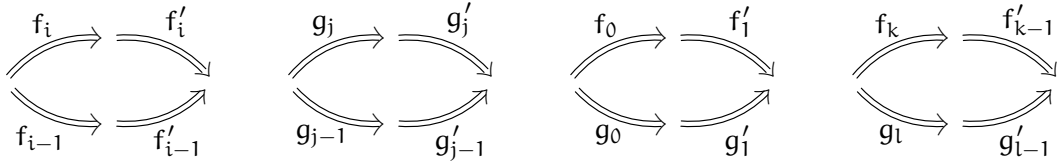
\[
\xymatrix @R=0.9em @C=3em {
& 
    \ar@2@/^/ [dr] ^-{f'_i}
\\
	\ar@2@/^/ [ur] ^-{f_i}
	\ar@2@/_/ [dr] _-{f_{i-1}}
&&
\\
& 
  \ar@2@/_/ [ur] _-{f'_{i-1}}
}
\qquad
\xymatrix @R=0.9em @C=3em {
& 
    \ar@2@/^/ [dr] ^-{g'_j}
\\
	\ar@2@/^/ [ur] ^-{g_j}
	\ar@2@/_/ [dr] _-{g_{j-1}}
&&
\\
& 
  \ar@2@/_/ [ur] _-{g'_{j-1}}
}
\qquad
\xymatrix @R=0.9em @C=3em {
& 
    \ar@2@/^/ [dr] ^-{f'_1}
\\
	\ar@2@/^/ [ur] ^-{f_0}
	\ar@2@/_/ [dr] _-{g_0}
&&
\\
& 
  \ar@2@/_/ [ur] _-{g'_1}
}
\qquad
\xymatrix @R=0.9em @C=3em {
& 
    \ar@2@/^/ [dr] ^-{f'_{k-1}}
\\
	\ar@2@/^/ [ur] ^-{f_k}
	\ar@2@/_/ [dr] _-{g_l}
&&
\\
& 
  \ar@2@/_/ [ur] _-{g'_{l-1}}
}
\]
with same $1$-target, for all $2\leq i \leq k-1$ and $2\leq j \leq l-1$. Note that some of these $2$-spheres can be trivial. Then the $2$-sphere $(f,g)$ can be filled up by these $2$-spheres as follows:
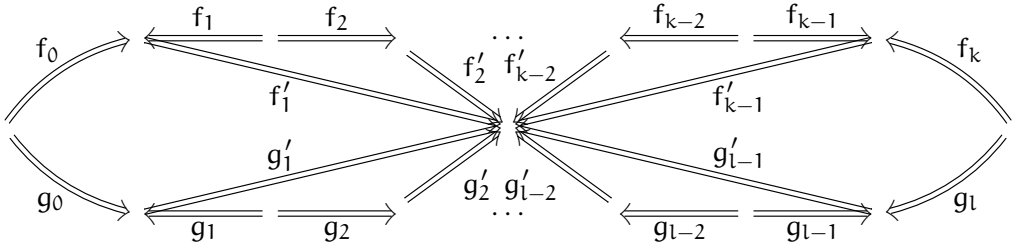
\[
\xymatrix@C=4em @R=2.25em{
&
  \ar@2 [drr] _(0.4){f'_1}
&
  \ar@2 [l] _-{f_1}
  \ar@2 [r] ^-{f_2}
&
\qquad\quad\cdots\quad\qquad
  \ar@2 "1,4"!<-45pt,0pt>;"2,4"!<0pt,0pt> ^(0.6){f'_2}
  \ar@2 "1,4"!<+45pt,0pt>;"2,4"!<0pt,0pt> _(0.6){f'_{k-2}}
&
  \ar@2 [l] _-{f_{k-2}}
  \ar@2 [r] ^-{f_{k-1}}
&
  \ar@2 [dll] ^(0.4){f'_{k-1}}
\\
  \ar@2 @/_/ [dr] _-{g_0}
  \ar@2 @/^/ [ur] ^-{f_0}
&&&&&&
  \ar@2 @/^/  [dl] ^-{g_l}
  \ar@2 @/_/  [ul] _-{f_k}
\\
&
  \ar@2 [urr] ^(0.4){g'_1}
&
  \ar@2 [l] ^-{g_1}
  \ar@2 [r] _-{g_2}
&
\qquad\quad\cdots\quad\qquad
  \ar@2 "3,4"!<-45pt,0pt>;"2,4"!<0pt,0pt> _(0.6){g'_2}
  \ar@2 "3,4"!<+45pt,0pt>;"2,4"!<0pt,0pt> ^(0.6){g'_{l-2}}
&
  \ar@2 [l] ^-{g_{l-2}}
  \ar@2 [r] _-{g_{l-1}}
&
  \ar@2 [ull] _(0.4){g'_{l-1}}
}
\]
By Lemma \ref{Lemma1}, these $2$-spheres can be filled up by $3$-cells of the $(3,1)$-category $\tck{\Sr^{sd}(\Sigma,\psi)}$. Finally, the composition of these $3$-cells gives
a 3-cell of $\tck{\Sr^{sd}(\Sigma,\psi)}$ from $f$ to~$g$.
\end{proof}

Strict decreasingness is a required condition in Theorem \ref{Theorem:MainResult} as shown by the following example.

\subsubsection{Example}
Consider the $2$-polygraph $\Sigma$ without $2$-loop and containing two families $(f^i_j)_{i, j \in \mathbb{N}, ij=0}$ and $(g^i_j)_{i, j \in \mathbb{N}, ij=0}$ of $2$-cells satisfying the following conditions:
\begin{enumerate}[{\bf i)}]
\item the sequences $(f^0_n)_{n \in \mathbb{N}}$, $(f^n_0)_{n \in \mathbb{N}}$, $(g^0_n)_{n \in \mathbb{N}}$ and $(g^n_0)_{n \in \mathbb{N}}$ are infinite rewriting paths,
\item for any odd integer $n$, we have $t_1(f^n_0)=t_1(g^n_0)$ and $t_1(f^0_n)=t_1(g^0_n)$,
\item for any even integer $n$, we have $t_1(f^n_0)=t_1(f^0_n)$ and $t_1(g^n_0)=t_1(g^0_n)$,
\end{enumerate}
as indicated in the following diagram
\[
\scalebox{0.85}{
\xymatrix@C=2.5em @R=2.5em{
&&&&&
\\
&&&&&
\\
\quad
&
  \ar@2 @/^5ex/ [uurrr] ^-{f_0^0}
  \ar@2 @/_5ex/ [ddrrr] _-{g_0^0}
&
\quad
  \ar@{<-}@2 @/^3ex/ [uurr] ^(0.3){f_0^1}
  \ar@{<-}@2 @/_3ex/ [ddrr] _(0.3){g_0^1}
  \ar@2  @/^2ex/ [urr] ^(0.7){f_0^2}
  \ar@2 @/_2ex/ [drr] _(0.7){g_0^2}
&
\quad
  \ar@{<-}@2 @/^1ex/ [ur] ^(0.23){f_0^3}
  \ar@{<-}@2 @/_1ex/ [dr] _(0.23){g_0^3}
&
\cdots
&
\quad
  \ar@{<-}@2 @/_1ex/ [ul] _(0.23){f_3^0}
  \ar@{<-}@2 @/^1ex/ [dl] ^(0.23){g_3^0}
&
\quad
  \ar@{<-}@2 @/_3ex/ [uull] _-{f_1^0}
  \ar@{<-}@2 @/^3ex/ [ddll] ^-{g_1^0}
  \ar@2 @/_2ex/ [ull] _(0.7){f_2^0}
  \ar@2 @/^2ex/ [dll] ^(0.7){g_2^0}
\\
&&&&&
\\
&&&&&
}
}
\]
and such that the only critical branchings of $\Sigma$ are of one of the following forms:
\[
(f^n_0,g^n_0), \; (f^0_n,g^0_n),\; \text{for $n$ even, and \,}
(f^0_n,f^n_0), \; (g^0_n,g^n_0),\; \text{for $n$ odd.}
\]
Let us consider the globular extension $\Gamma$ of the free $(2,1)$-polygraph $\tck{\Sigma}_2$, defined by the following infinite family of $2$-spheres: 
\[
(f^n_0 \star_1 f^{n+1}_0,f^0_n \star_1 f^0_{n+1})
\quad \text{and}\quad 
(g^n_0 \star_1 g^{n+1}_0,g^0_n \star_1 g^0_{n+1})
\quad \text{for $n$ odd,}
\]
and
\[
(f^n_0 \star_1 f^{n+1}_0,g^n_0 \star_1 g^{n+1}_0)
\quad \text{and} \quad
(f^0_n \star_1 f^0_{n+1},g^0_n \star_1 g^0_{n+1})
\quad \text{for $n$ even.}
\] 
The globular extension $\Gamma$ contains one generating confluence for each critical branching of $\Sigma$. However, we cannot define a $3$-cell in the free $(3,1)$-category generated by $(\Sigma,\Gamma)$ with $2$-source  $f^0_0 \star_1 f^0_1$ and \linebreak $2$-target $g^0_0 \star_1 g^0_1$. As a consequence, $\Gamma$ does not form a homotopy basis of the $(2,1)$-category $\tck{\Sigma}_2$.
In fact, we note that the $2$-polygraph $\Sigma$ is not strictly decreasing, because no labelling of $\Sigma$ is well-founded, but decreasing with the singleton labelling.

Following \ref{Subsubsection:DecreasingnessFromQuasiTermination}, any quasi-convergent $2$-polygraph $\Sigma$ is strictly decreasing with respect to any quasi-normal form labelling $\labQNF$. The following result is a consequence of Theorem \ref{Theorem:MainResult}.

\begin{corollary}
Let $\Sigma$ be a quasi-convergent $2$-polygraph and let $\labQNF$ be a quasi-normal form labelling of $\Sigma$.
Let $\Sr^{sd}(\Sigma,\labQNF)$ be a strictly decreasing Squier's completion of $\Sigma$.
If the labelling $\labQNF$ is compatible with contexts and $(\Sigma,\labQNF)$ is Peiffer decreasing with respect to the extension $\Sr^{sd}(\Sigma,\labQNF)$, then $\Sr^{sd}(\Sigma,\labQNF)$ is a coherent presentation of the category presented by $\Sigma$.
\end{corollary}

\subsubsection{Example}
\label{Example:homotopyBasisSigmaB3+}
By Theorem \ref{Theorem:MainResult}, the five $3$-cells given in 
\ref{Example:vOSCompletionSigmaB3+} form a homotopy basis of the \linebreak $2$-polygraph $\sigmab$. Indeed, the $2$-polygraph $\sigmab$ is strictly decreasing for the labelling QNF defined in Example~\ref{Example:QNFSigmaB3+}. This labelling being compatible with contexts, the only remaining point concerns the Peiffer confluences. Let us show that any Peiffer confluence is equivalent to a decreasing confluence diagram.
Consider a Peiffer branching $(f v, u g) : uv \dfl (u'v,uv')$ of $\sigmab$ and its Peiffer confluence $(u'g,fv') : (u'v,uv') \dfl u'v'$:
\[
\xymatrix @R=0.15em @C=2.5em {
& u'v
	\ar@2@/^/ [dr] ^{u' g}
\\
uv
	\ar@2@/^/ [ur] ^-{fv}
	\ar@2@/_/ [dr] _-{ug}
&& u'v'
\\
& uv'
	\ar@2@/_/ [ur] _-{fv'}
}
\]
By definition of $\sigmab$, there exist rewriting steps $f':u'\dfl u$ and $g':v'\dfl v$. It follows that this Peiffer confluence is equivalent with respect to $\Lr(\sigmab)$ to each of the following Peiffer confluence:
\[
(f'v'\cdot u g', u'g'\cdot f'v),
\qquad
(u'g\cdot f'v',f'v\cdot ug),
\qquad
(ug'\cdot fv,fv'\cdot u'g').
\]
The equivalences are proved by the following diagrams:
\[
\xymatrix @R=1.9em @C=3em {
& uv'
	\ar@2@/^/ [dr] ^{u g'}
	\ar@2@/^/ [dl] ^-{fv'}
\\
u'v'
	\ar@2@/^/ [ur] ^-{f'v'}
	\ar@2@/_/ [dr] _-{u'g'}
&& uv
	\ar@2@/^/ [ul] ^-{ug}
	\ar@2@/_/ [dl] _-{fv}
\\
& u'v
	\ar@2@/_/ [ur] _-{f'v}
	\ar@2@/_/ [ul] _-{u'g}
}
\qquad
\xymatrix @R=1.5em @C=3em  {
& u'v'
	\ar@2@/^/ [dr] ^{f'v'}
\\
u'v
	\ar@2@/^/ [ur] ^-{u'g}
	\ar@2@/_/ [dr] _-{f'v}
&& 
uv'
  \ar@2@/^/ [ul] ^-{fv'}
\\
& uv
	\ar@2@/_/ [ur] _-{ug}
	\ar@2@/_/ [ul] _-{fv}
}
\qquad
\xymatrix @R=1.5em @C=3em  {
& uv
	\ar@2@/^/ [dr] ^{fv}
	 \ar@2@/^/ [dl] ^-{ug}
\\
uv'
	\ar@2@/^/ [ur] ^-{ug'}
	\ar@2@/_/ [dr] _-{fv'}
&& 
u'v
  \ar@2@/_/ [dl] _-{u'g}
\\
& u'v'
	\ar@2@/_/ [ur] _-{u'g'}
}
\]
Finally, in each family of such four Peiffer confluences, one of them is decreasing with respect to the labelling $\labQNF$.

\subsubsection{Decreasingness from termination}
Given a confluent and terminating $2$-polygraph $\Sigma$, any \linebreak $1$-cell $u$ of $\Sigma_1^\ast$ has a unique normal form denoted by $\rep{u}$.
We define the \emph{labelling to the normal form} $\labNF:\etapes{\Sigma} \fl \Sigma_1^\ast$ by setting for each rewriting step $f$, $\labNF(f) = t_1(f)$. We choose on $\Sigma_1^\ast$ the order induced by the rewrite relation defined by $\Sigma_2$. This labelling is compatible with contexts and makes the $2$-polygraph $\Sigma$ strictly decreasing and Peiffer decreasing. Moreover, $\Sigma$ being terminating it does not have loop and in particular the decreasing Squier completion coincides with the Squier completion.  
In this way, the Squier coherence theorem obtained for convergent string rewriting systems in \cite{Squier94} is a consequence of Theorem~\ref{Theorem:MainResult}:

\begin{corollary}[{\cite[Theorem 5.2]{Squier94}}]
Let $\Sigma$ be a convergent $2$-polygraph. Any Squier's completion~$\Sr(\Sigma)$ of $\Sigma$ is a coherent presentation of the category presented by $\Sigma$.
\end{corollary}

\subsection{Finiteness homotopical and homological conditions by decreasingness}

\subsubsection{Finite derivation type}
A $2$-polygraph $\Sigma$ has \emph{finite derivation type}, FDT for short, if the free $(2,1)$-category $\tck{\Sigma}_2$ has a finite homotopy basis, see {\cite[Section 4]{GuiraudMalbos16}}. Squier proved that this property is invariant for finite string rewriting systems: if $\Sigma$ and $\Upsilon$ are two finite $2$-polygraphs, then $\Sigma$ has FDT if and only if $\Upsilon$ has FDT. As a consequence, the property can be defined on finitely presented monoids: a finitely presented monoid has FDT if it has a presentation by a $2$-polygraph that has FDT.

For a convergent $2$-polygraph~$\Sigma$, its is well known that a family of generating confluences forms a homotopy basis of $\tck{\Sigma}_2$. A finite convergent $2$-polygraph having a finite number of critical branchings, then it has FDT. However, a finite decreasing $2$-polygraph can have an infinite decreasing Squier's completion. Indeed, the set of decreasing confluences is always finite for a finite $2$-polygraph but the set of elementary $2$-loops may be infinite. 
As a consequence of Theorem~\ref{Theorem:MainResult} we can formulate the following result.

\begin{proposition}
\label{Proposition:TDFdecreasing}
Let $(\Sigma,\psi)$ be a strictly decreasing and quasi-convergent $2$-polygraph such that the labelling $\psi$ is compatible with contexts and Peiffer decreasing. If $\Sigma$ has a finite set of $2$-cells and a finite set of elementary $2$-loops, then it has finite derivation type.
\end{proposition}

\begin{example}
\label{Example:notStable2}
Let us consider the $2$-polygraph $\Sigma$ with only one $0$-cell, $\Sigma_1=\{a,b,c,d,d'\}$ and $\Sigma_2=\{ab \dfl a, ac \dfl da, da \dfl d'a, d'a \dfl ac\}$.
This $2$-polygraph presents a monoid which has not FDT, see {\cite[Section 5]{Lafont95}}. Moreover, it has only one elementary $2$-loop up to equivalence and a finite number of critical branchings. As a consequence, there is no well-founded labelling compatible with contexts  making the $2$-polygraph $\Sigma$ strictly decreasing and Peiffer decreasing.
\end{example}

\subsubsection{Finite homological type $FP_3$ by decreasingness}
As a final remark, let us mention another application to computation of low-dimensional homological properties of monoids.
Let $\M$ be a monoid and $\Sigma$ be a coherent presentation of $\M$.
Following {\cite[Proposition 5.3.2.]{GuiraudMalbos16}}, there is a partial resolution 
\[
\mathbb{Z}\M[\Sigma_3] \ofl{d_3} \mathbb{Z}\M[\Sigma_2] \ofl{d_2} \mathbb{Z}\M[\Sigma_1] \ofl{d_1} \mathbb{Z}\M \ofl{\epsilon} \mathbb{Z} \longrightarrow 0
\]
of left-modules over the free ring $\mathbb{Z}\M$ over $\M$, where 
$\mathbb{Z}$ denotes the trivial $\mathbb{Z}\M$-module and $\mathbb{Z}\M[\Sigma_i]$ denotes the free $\mathbb{Z}\M$-module generated by~$\Sigma_i$. The morphisms of $\mathbb{Z}\M$-modules are defined by $\epsilon(u)=1$, for any $u$ in~$\M$, and
$d_1$, $d_2$ and $d_3$ are defined on the generators by
\[
d_1(x) = x - 1,
\qquad
d_2(\alpha)=[s_1(\alpha)]-[t_1(\alpha)],
\qquad
d_3(A)=[s_2(A)]-[t_2(A)],
\]
for any $x$ in $\Sigma_1$, $\alpha$ in $\Sigma_2$ and $A$ in $\Sigma_3$ , and with the bracket notations of {\cite[Section 5]{GuiraudMalbos16}}.

In particular, by Theorem~\ref{Theorem:MainResult}, if $(\Sigma,\psi)$ is a strictly decreasing $2$-polygraph such that $\psi$ is compatible with contexts  and Peiffer decreasing, the coherent presentation given by the strictly decreasing Squier completion $\Sr^{sd}(\Sigma,\psi)$ induces such a partial resolution. If moreover $\Sigma$ has a finite set of $2$-cells and a finite set of elementary $2$-loops, then it has finite homological type~$FP_3$. We expect that our construction can be extended in higher-dimension of homology producing infinite lenght resolutions for monoids presented by quasi-convergent presentations, and thus weakening the termination hypothesis required in construction of such resolutions as in \cite{Kobayashi90, Anick86}.

\subsubsection{Example}
Following Example \ref{Example:homotopyBasisSigmaB3+}, the monoid $\B_3^+$, admits a coherent presentation with two 1-cells $s$ and $t$, two $2$-cells $\alpha : sts \dfl tst$ and $\beta : tst \dfl sts$ and the five $3$-cells $D_{t\alpha ,\beta s}^{\labQNF}$, $D_{\beta s,t\alpha}^{\labQNF}$, $D_{\alpha ts,st\alpha}^{\labQNF}$, $D_{\beta st,ts\beta}^{\labQNF}$, $E_{\alpha ,\beta}$.
Using the homotopical reduction procedure introduced in {\cite[2.3.1.]{GaussentGuiraudMalbos15}} with a collapsible part made of the $3$-cell $E_{\alpha,\beta}$, we can reduce this coherent presentation to a coherent presentation of the monoid $\B_3^+$ with the same $k$-cells for $k\leq 2$ and with no $3$-cells.
Hence, we obtain the following resolution 
\[
0 \fll \mathbb{Z}\M[\alpha,\beta] \ofl{d_2} \mathbb{Z}\M[s,t] \ofl{d_1} \mathbb{Z}\M \ofl{\epsilon} \mathbb{Z} \longrightarrow 0.
\]
We deduce the homology of the monoid $\B_3^+$ with integral coefficients: $\mathrm{H}_n(\M,\mathbb{Z})~=~\mathbb{Z}$ for $n\leq 2$ and $\mathrm{H}_n(\M,\mathbb{Z})=0$, for $n\geq 3$.

\begin{small}
\renewcommand{\refname}{\Large\textsc{References}}
\bibliographystyle{plain}
\bibliography{biblioCURRENT}
\end{small}

\vspace{1cm}

\quad

\begin{small}
\noindent \textsc{Cl\'ement Alleaume} \\
\url{clement.alleaume@univ-st-etienne.fr} \\
Univ Lyon, Universit\'e Claude Bernard Lyon 1\\
CNRS UMR 5208, Institut Camille Jordan\\
43 blvd. du 11 novembre 1918\\
F-69622 Villeurbanne cedex, France
\end{small}

\vskip+5pt

\begin{small}
\noindent \textsc{Philippe Malbos} \\
\url{malbos@math.univ-lyon1.fr} \\
Univ Lyon, Universit\'e Claude Bernard Lyon 1\\
CNRS UMR 5208, Institut Camille Jordan\\
43 blvd. du 11 novembre 1918\\
F-69622 Villeurbanne cedex, France
\end{small}

\vspace{1cm}

\hfill \begin{small}---\;\;\today\;\;-\;\;\hhmm\;\;---\end{small}

\end{document}